\documentclass[a4paper,11pt]{scrartcl}
\usepackage{etex}
\usepackage{amssymb}
\usepackage{amsmath}
\usepackage{amssymb}
\usepackage{amsthm}
\usepackage{url}

\usepackage{mathtools}
\usepackage{stmaryrd}
\usepackage{color}
\usepackage{comment}

\usepackage[all]{xy}
\SelectTips{cm}{}

\newcommand{\newpara}{\vspace{1ex}}

\theoremstyle{theorem}
  \newtheorem{theorem}{Theorem}[section]
  \newtheorem{proposition}[theorem]{Proposition}
  \newtheorem{lemma}[theorem]{Lemma}
  \newtheorem{corollary}[theorem]{Corollary}
\theoremstyle{definition}
  \newtheorem{definition}[theorem]{Definition}
  \newtheorem{example}[theorem]{Example}
  \newtheorem{remark}[theorem]{Remark}
  \newtheorem{notation}[theorem]{Notation}

\newcommand{\bbA}{\mathbb{A}}
\newcommand{\bbB}{\mathbb{B}}
\newcommand{\bbN}{\mathbb{N}}

\newcommand{\bbR}{\mathbb{R}}

\newcommand{\bbS}{\mathbb{S}}
\newcommand{\bbP}{\mathbb{P}}
\newcommand{\bbH}{\mathbb{H}}

\newcommand{\calB}{\mathcal{B}}
\newcommand{\calC}{\mathcal{C}}

\newcommand{\calF}{\mathcal{F}}
\newcommand{\calG}{\mathcal{G}}

\newcommand{\calU}{\mathcal{U}}

\newcommand{\cat}[1]{\mathcal{#1}}

\newcommand{\bigland}{\bigwedge}

\newcommand{\pow}{\mathcal{P}}

\newcommand{\id}{\mathrm{id}}
\newcommand{\isom}{\simeq}

\newcommand{\Term}{\mathcal{T}}
\newcommand{\Free}{\mathcal{F}}

\newcommand{\set}[2]{\left\{\, #1 \mathrel{}\middle|\mathrel{} #2 \,\right\}}

\newcommand{\abs}[1]{{\lvert #1 \rvert}}
\newcommand{\norm}[1]{\left\lVert #1 \right\rVert}
\newcommand{\blank}{{-}} 

\newcommand{\mono}{\rightarrowtail}
\newcommand{\epi}{\twoheadrightarrow}

\newcommand{\onto}{\epi}

\newcommand{\e}{\varepsilon}
\newcommand{\sem}[1]{{\llbracket #1 \rrbracket}}
\newcommand{\bbRp}{\bbR_{\ge 0}}
\newcommand{\quot}{/}

\newcommand{\cle}{\preccurlyeq}
\newcommand{\cge}{\succcurlyeq}
\newcommand{\cjoin}{\curlyvee}
\newcommand{\cmeet}{\curlywedge}
\newcommand{\bigcjoin}{\bigcurlyvee}

\newcommand{\Con}{\mathop{\mathbf{Con}}\nolimits}

\mathcode`\<="4268 
\mathcode`\>="5269 
\mathchardef\gt="313E 
\mathchardef\lt="313C 
\renewcommand{\phi}{\varphi}

\newcommand{\im}{\mathop{\mathrm{im}}\nolimits}
\renewcommand{\subset}{\subseteq}

\title{Continuous Varieties of Metric and Quantitative Algebras}
\author{Wataru Hino \\ The University of Tokyo, Japan \\
        \url{wataru@is.s.u-tokyo.ac.jp}}
\date{\today}

\begin{document}

\maketitle

\begin{abstract}
  A \emph{metric algebra} is a metric variant of the notion of $\Sigma$-algebra,
  first introduced in universal algebra to deal with
  algebras equipped with metric structures such as normed vector spaces.
  In this paper, we showed metric versions of the \emph{variety theorem},
  which characterizes \emph{strict varieties}
  (classes of metric algebras defined by \emph{metric equations})
  and \emph{continuous varieties}
  (classes defined by \emph{a continuous family of basic quantitative inferences})
  by means of closure properties.
  To this aim, we introduce the notion of \emph{congruential pseudometric}
  on a metric algebra, which corresponds to congruence in classical universal algebra, and
  we investigate its structure.
\end{abstract}

\section{Introduction}
\subsection{Metric and Quantitative Algebra}
A \emph{quantitative algebra} is introduced by Mardare et al.\
as a quantitative variant of the notion of $\Sigma$-algebra, in the sense of classical universal algebra in \cite{Mardare2016}.
They use an atomic formula of the form $s =_{\e} t$, where $\e$ is a non-negative real number,
instead of an equation $s = t$, and give a complete deductive system with respect
to quantitative algebras.
They investigate classes defined by \emph{basic quantitative inferences},
which are formulas of the form $\bigland_{i = 1}^{n} x_i =_{\e_i} y_i \to s =_{\e} t$
where $x_i$ and $y_i$ are restricted to variables.
They show that various well-known metric constructions, such as the Hausdorff metric,
the Kantorovich metric and the Wasserstein metric, naturally arise as free quantitative algebras
with suitable axioms consisting of basic quantitative inferences.
The theory of quantitative algebra is applied
to the axiomatization of the behavioral distance \cite{Bacci2016}.

In fact, the idea of using indexed binary relations to axiomatize metric structures
is already in the literature of universal algebra
\cite{Weaver1995, Khudyakov2003} under the name of \emph{metric algebra}.
This notion is slightly wider than that of quantitative algebra
in the sense that operations in metric algebras are not required to be non-expansive.
Weaver \cite{Weaver1995} and Khudyakov \cite{Khudyakov2003} prove
continuous versions of the characterization theorem
for \emph{quasivarieties}, i.e., classes of algebras defined by implications,
and the decomposition theorem corresponding to the one in the classical theory.

However a metric version of the variety theorem has been missing for long.
We give a very straightforward version in \cite{Hino2016},
and Mardare et al.\ \cite{Mardare2017} give the characterization theorem for $\kappa$-variety,
where $\kappa$ is a cardinal, which generalizes our result in \cite{Hino2016}.

\subsection{Contributions}
In this paper, we will investigate the universal algebraic treatment of metric
and quantitative algebra. More specifically,

\begin{itemize}
  \item We give a clean formulation of the theory of metric
    and quantitative algebra based
    on \emph{congruential pseudometric} (Definition~\ref{def:cong}).
    We prove some basic results on congruential pseudometric,
    including the metric variant of the isomorphism theorems.
    Especially the characterization theorem of direct products via congruential pseudometrics
    seems non-trivial since we need to assume that the given metric space is complete.
  \item We prove the variety theorem for classes of metric (or quantitative) algebras
    defined by metric equations.
    This is proved in our previous work \cite{Hino2016}.
    Here we give a more concise proof by congruential pseudometrics.
  \item We prove the variety theorem for \emph{continuous varieties},
  which are classes of metric (or quantitative) algebras defined by basic quantitative inferences
  and satisfy the \emph{continuity condition}.
\end{itemize}

As we mentioned, a basic quantitative inference is an implicational formula
whose assumptions are metric equations between variables.
One of the main challenges when considering implicationally defined classes is the \emph{size problem};
it is often easy to show that a given class is defined by implications
if we allow infinitely many assumptions, and difficulties arise
when we want to have finitary axioms.
We use \emph{ultraproduct} to deal with the size problem, following the approach in \cite{Weaver1995},
but we make the relation between ultraproducts and the size problem more explicit:
we first show the weak version of the compactness theorem for metric algebras,
and use it for the restriction of the size of assumptions.

A variety theorem for
$\kappa$-variety\footnote{a class defined by $\kappa$-basic quantitative inferences.}
is already shown in \cite{Mardare2017}, but it lacks the continuity condition.
The continuity condition is important, especially when we work with complete metric spaces.
Indeed, as pointed out in \cite{Mardare2016}, a class defined by basic quantitative inferences
is closed under completion if its axioms satisfy the continuity condition
(the situation is the same for quasivarieties \cite{Weaver1995}).
Moreover the continuity condition also implies the closure property
under \emph{ultralimits},
which can be seen as a robustness condition in some sense.
This point is discussed in Section~\ref{sec:ultralimit-metsp}.

\section{Preliminaries}
\label{ch:preliminaries}

In this section, we review some notions that we will need in the following sections.

\subsection{Classical Universal Algebra}
\label{sec:classical-theory}

Let $\Sigma$ be an algebraic signature, i.e., a set with an arity map
$\abs{\blank} \colon \Sigma \to \bbN$.
We define $\Sigma_n$ by $\Sigma_n = \set{\sigma \in \Sigma}{\abs{\sigma} = n}$
for each $n \in \bbN$.

\begin{definition}[See e.g.\ \cite{Burris1981}] \mbox{}
  \begin{itemize}
    \item A \emph{$\Sigma$-algebra} is a tuple $A = (A, (\sigma^A)_{\sigma \in \Sigma})$
    where $A$ is a set endowed with an operation $\sigma^A \colon A^n \to A$ for
    each $\sigma \in \Sigma_n$. We will just write $\sigma$ for $\sigma^A$
    if $A$ is clear from the context.

    \item A map $f \colon A \to B$ between $\Sigma$-algebras is a \emph{$\Sigma$-homomorphism}
    if it preserves all $\Sigma$-operations, i.e., \
    $f(\sigma^A(a_1, \ldots, a_n)) = \sigma^B(f(a_1), \ldots, f(a_n))$
    for each $\sigma \in \Sigma_n$ and $a_1, \ldots, a_n \in A$.

    \item A \emph{subalgebra} of a $\Sigma$-algebra $A$ is a subset of $A$ closed under $\Sigma$-operations,
    regarded as a $\Sigma$-algebra by restricting operations.
    A subalgebra is identified with (the isomorphic class of) a pair $(B, i)$,
    where $B$ is a $\Sigma$-algebra
    and $i \colon B \mono A$ is an injective homomorphism.

    \item The \emph{product} of $\Sigma$-algebras $(A_i)_{i \in I}$ is the direct product of
    the underlying sets endowed with the pointwise $\Sigma$-operations.

    \item A \emph{quotient} (also called a \emph{homomorphic image}) of a $\Sigma$-algebra $A$ is
    a pair $(B, \pi)$ where $B$ is a $\Sigma$-algebra
    and $\pi \colon A \epi B$ is a surjective homomorphism.

    \item Given a set $X$, the \emph{set $\Term_{\Sigma} X$ of $\Sigma$-terms over $X$} is inductively defined as follows:
    each $x \in X$ is a $\Sigma$-term (called a \emph{variable}), and
    if $\sigma \in \Sigma_n$ and $t_1, \ldots, t_n$ are $\Sigma$-terms,
    then $\sigma(t_1, \ldots, t_n)$ is a $\Sigma$-term.

    The set $\Term_{\Sigma} X$ is endowed with a natural $\Sigma$-algebra structure,
    and this algebra is called the \emph{free $\Sigma$-algebra over $X$}.
    It satisfies the following universality:
    for each $\Sigma$-algebra $A$, a map $v \colon X \to A$ uniquely extends
    to a $\Sigma$-homomorphism $v^{\sharp} \colon \Term_{\Sigma} X \to A$.
    We also denote $v^{\sharp}(t)$ by $\sem{t}_v$.

    \item Given a set $X$,
    a \emph{$\Sigma$-equation over $X$} is a formula $s = t$ where $s, t \in \Term_{\Sigma} X$.

    We say that a $\Sigma$-algebra $A$ \emph{satisfies} a $\Sigma$-equation $s = t$ over $X$
    (denoted by $A \models s = t$)
    if $\sem{s}_v = \sem{t}_v$ holds for any map $v \colon X \to A$.

    For a set of $\Sigma$-equations, we say $A \models E$ if
    $A$ satisfies all equations in $E$.

    \item A class $\cat{K}$ of $\Sigma$-algebras is a \emph{variety} if there is a set $E$ of
    equations such that
    $\cat{K} = \set{A \colon \text{a $\Sigma$-algebra}}{A \models E}$ holds.
  \end{itemize}
\end{definition}

If the signature $\Sigma$ is obvious from the context,
we omit the prefix $\Sigma$ and just say \emph{homomorphism}, \emph{equation}, etc.

\newpara

The following theorem is fundamental in universal algebra, and is proved by Birkhoff.
It states that the property of being a variety is equivalent to a certain closure property;
see e.g.\ \cite{Burris1981}.
Our main goal is to prove the metric version of this theorem.

\begin{theorem}[Variety theorem \cite{Birkhoff1935}]
  A class $\cat{K}$ of $\Sigma$-algebras is a variety
  if and only if $\cat{K}$ is closed under subalgebras, products and quotients.
\end{theorem}

\subsection{Metric Space and Pseudometric}
\label{sec:metric-sp}

Now we review the notions regarding metric spaces.

\begin{definition}[e.g.\ \cite{Howes1995}] \mbox{}
  \begin{itemize}
    \item An \emph{extended real} is an element of $\overline{\bbR} = \bbR \cup \{\pm \infty\}$.

    \item Given a set $X$, an \emph{(extended) pseudometric} on $X$ is
    a function $d \colon X \times X \to [0, \infty]$ that satisfies
    $d(x, x) = 0$, $d(x, y) = d(y, x)$ and $d(x, y) + d(y, z) \ge d(x, z)$.
    A pseudometric $d$ is a \emph{metric} if
    it also satisfies $d(x, y) = 0 \Rightarrow x = y$.

    A pseudometric space (resp. metric space) is a tuple $(X, d)$
    where $X$ is a set and $d$ is a pseudometric (resp. metric) on $X$.

    \item A map $f \colon X \to Y$ between metric spaces is \emph{non-expansive}
    if it satisfies $d(f(x), f(y)) \le d(x, y)$ for each $x, y \in X$.

    \item For a family $(X_i, d_i)_{i \in I}$ of metric spaces,
    its \emph{product} is defined by $\left( \prod_{i} X_i, d \right)$ where
     $d((x_i)_i, (y_i)_i) = \sup_{i \in I} d_i(x_i, y_i)$,
     and $d$ is called the \emph{supremum metric}.
  \end{itemize}
\end{definition}

Note that we admit infinite distances, called \emph{extended},
because the category of extended metric spaces is categorically more amenable
than that of ordinary metric spaces; it has coproducts and arbitrary products.
Moreover a set can be regarded as a discrete metric space,
where every pair of two distinct points has an infinite distance.

\newpara

In this paper, we denote $d(x, y) \le \e$ by $x =_{\e} y$.
To consider a metric structure as a family of binary relations
works well with various metric notions;
e.g.\ $f \colon X \to Y$ is non-expansive if and only if $x =_\e y$ implies
$f(x) =_\e f(y)$ for each $x, y \in X$ and $\e \ge 0$.
The supremum metric of the product space $\prod_{i \in I} X_i$ is also compatible with
this relational view of metric spaces; it is characterized by
$(x_i)_i =_\e (y_i)_i \iff x_i =_{\e} y_i$ for all $i \in I$.

We adopt the supremum metric rather than other metrics
(e.g.\ the 2-product metric) for the product of metric spaces.
One reason is the compatibility with the relational view above.
Another reason is that it corresponds to the product
in the category of extended metric spaces and non-expansive maps.

Recall that the supremum metric does not always give rise to the product topology;
the product of uncountably many metrizable spaces is not in general metrizable.

\newpara

Given a pseudometric space, we can always turn it into a metric space
by identifying points whose distance is zero.

\begin{proposition}[e.g.\ \cite{Howes1995}]
  \label{thm:metric-id}
  Given a pseudometric $d$ on $X$, the binary relation $\sim_d$ on $X$
  defined by $x \sim_d y \Leftrightarrow d(x, y) = 0$ is an equivalence relation.
  Moreover $\overline{d}([x], [y]) = d(x, y)$ defines
  a metric $\overline{d}$ on $\overline{X} = X \quot {\sim_d}$
  and yields to a metric space $(\overline{X}, \overline{d})$.
\end{proposition}

\begin{definition}[e.g.\ \cite{Howes1995}]
  \label{def:metric-id}
  The equivalence relation $\sim_d$ defined in Proposition~\ref{thm:metric-id}
  is called the \emph{metric identification of $d$},
  and $(\overline{X}, \overline{d})$ is called a \emph{metric space induced by the pseudometric $d$}.
  We denote it by $X \quot d$.
\end{definition}

Technically, whenever we encounter a pseudometric space,
we can regard it as a metric space by the above construction.
However it does not mean that pseudometric is a totally redundant notion.
Our slogan is: \emph{pseudometrics is to metric spaces what equivalence relations is to sets}.
Later we discuss pseudometrics that are compatible with given algebraic structures,
which correspond to congruences in classical universal algebra.
We utilize this notion intensively in the proof of the variety theorem.

\subsection{Filter and Limit}
\label{sec:filter-limit}

Limits with respect to filters play an important role
in the construction of ultralimits of metric spaces.
Most of the results are straightforward generalizations of those for the usual limits.

\begin{definition}[e.g.\ \cite{Jech2002}]
  Let $I$ be a nonempty set.
  A \emph{filter on $I$} is a subset $\calF$ of $\pow(I)$
  that satisfies the following conditions:
  \begin{enumerate}
    \item $I \in \calF$,\quad $\emptyset \not\in \calF$.
    \item If $G \in \calF$ and $G \subset H$, then $H \in \calF$.
    \item If $G \in \calF$ and $H \in \calF$, then $G \cap H \in \calF$.
  \end{enumerate}
  A filter $\calF$ is an \emph{ultrafilter} if, for any $G \subset I$,
  either $G \in \calF$ or $I \setminus G \in \calF$ holds.
\end{definition}

\begin{example}
  Let $I$ be a nonempty set.
  \begin{itemize}
    \item For $a \in I$,
    a set $\calF_{a}$ defined by
    $\calF_{a} = \set{G \subset I}{a \in G}$ is an ultrafilter on $I$.
    It is called the \emph{principal ultrafilter at $a$}.
    \item Assume $I$ is infinite.
    The set $\calF_{\omega}$ of cofinite (i.e.\ its complement is finite)
    subsets of $I$ is a filter on $I$.
    It is called the \emph{cofinite filter on $I$}.
    A filter is \emph{free} if it contains the cofinite filter.
  \end{itemize}
\end{example}

\begin{lemma}[e.g.\ \cite{Jech2002}] \mbox{}
  \begin{enumerate}
    \item For a filter $\calF$ on $I$ and $J \in \calF$,
    the family $\calF |_{J} := \calF \cap \pow(J)$ is a filter on $J$.
    If $\calF$ is an ultrafilter, then $\calF |_{J}$ is an ultrafilter.
    \item Let $\calU$ be an ultrafilter on $I$ and $A, B \subset I$.
    If $A \cup B \in \calU$ holds, then either $A \in \calU$ or $B \in \calU$ holds.
  \end{enumerate}
\end{lemma}

\begin{definition}
  Let $\calF$ be a filter on $I$.
  For an $I$-indexed family $(a_i)_{i \in I}$ of extended reals,
  we define $\liminf_{i \to \calF} a_i$ and $\limsup_{i \to \calF} a_i$ by
  \begin{align*}
    \liminf_{i \to \calF} a_i &= \adjustlimits \sup_{J \in \calF} \inf_{i \in J} a_i \\
    \limsup_{i \to \calF} a_i &= \adjustlimits \inf_{J \in \calF} \sup_{i \in J} a_i \quad.
  \end{align*}
  When $\liminf_{i \to \calF} a_i = \limsup_{i \to \calF} a_i = \alpha \in [-\infty, \infty]$,
  we write $\lim_{i \to \calF} a_i = \alpha$.
\end{definition}

\begin{example}
  \begin{itemize}
    \item For a set $I$ and $k \in I$, we have $\lim_{i \to \calF_k} a_i = a_k$.
    \item For the cofinite filter $\calF_\omega$ on $\bbN$,
    we have $\liminf_{n \to \calF_\omega} a_n = \liminf_{n \to \infty} a_n$
    and $\limsup_{n \to \calF_\omega} a_n = \limsup_{n \to \infty} a_n$.
    Thus the limit with respect to a filter is the generalization of the usual limit.
  \end{itemize}
\end{example}

The following results on filters and limits are all elementary.

\begin{lemma}
  \label{lem:filter-le}
  Let $\calF$ and $\calG$ be a filter on $I$ where $\calF \subset \calG$.
  For a family $(a_i)_{i \in I}$ of extended reals,
  $\limsup_{i \to \calF} a_i \ge \limsup_{i \to \calG} a_i$
  and $\liminf_{i \to \calF} a_i \le \liminf_{i \to \calG} a_i$.
\end{lemma}
\begin{proof}
  Obvious from the definition of limit infimum and supremum.
\end{proof}

\begin{lemma}
  \label{lem:sup-plus}
  Let $\calF$ be a filter on $I$ and $(x_i)_{i \in I}, (y_i)_{i \in I}$ be families of reals.
  Then we have:
  \[
    \limsup_{i \to \calF} (x_i + y_i) \le \limsup_{i \to \calF} x_i  + \limsup_{i \to \calF} y_i \quad .
  \]
\end{lemma}
\begin{proof}
  For $\e \gt 0$, there exist $J, J' \in \calF$ such that
  \[
    \sup_{i \in J} x_i + \sup_{i \in J'} y_i
    \le \limsup_{i \to \calF} x_i  + \limsup_{i \to \calF} y_i + \e \quad.
  \]
  Since
  $\sup_{i \in J} x_i + \sup_{i \in J'} y_i
  \ge \sup_{i \in J \cap J'} (x_i + y_i) \ge \limsup_{i \to \calF} (x_i + y_i)$,
  we have
  \[
  \limsup_{i \to \calF} (x_i + y_i) \le \limsup_{i \to \calF} x_i  + \limsup_{i \to \calF} y_i + \e \quad .
  \]
  Then letting $\e \to 0$ completes the proof.
\end{proof}

\begin{lemma}\label{lem:usual-limit}
  Let $\calF$ be a filter on $I$, and $(a_i)_{i \in I}$ be a family of extended reals.
  Then $\lim_{i \in \calF} a_i = \alpha$ if and only if
  $\set{i \in I}{\abs{a_i - \alpha} \le \e} \in \calF$ for any $\e \gt 0$.
\end{lemma}

\begin{proposition}
  \label{cont-limit}
  Let $\calF$ be a filter on $I$, $f \colon \bbR^n \to \bbR$ be a continuous function
  and $(x^0_i)_{i \in I}, \ldots, (x^n_i)_{i \in I}$ be families of real numbers.
  If $\lim_{i \to \calF} x^k_i = \alpha^k \in \bbR$ for each $k$,
  then $\lim_{i \to \calF} f(x^0_i, \ldots, x^n_i) = f(\alpha^0, \ldots, \alpha^n)$.
\end{proposition}
\begin{proof}
  Fix $\e \gt 0$. Since $f$ is continuous, there exists $\delta \gt 0$ such that
  for any $\vec{x} \in \bbR^n$ with $\abs{\vec{x} - \vec{\alpha}} \le \delta$,
  we have $\abs{f(\vec{x}) - f(\vec{\alpha})} \le \e$.
  Let $J = \set{i \in I}{\abs{\vec{x_i} - \vec{\alpha}} \le \delta}$
  and $J' = \set{i \in I}{\abs{f(\vec{x_i}) - f(\vec{\alpha})} \le \e}$.
  By Lemma~\ref{lem:usual-limit} we have $J \in \calF$, and
  since $J \subset J'$ holds, we also have $J' \in \cal{F}$.
  Again by Lemma~\ref{lem:usual-limit}, we conclude $\lim_{i \to \calF} f(\vec{x_i}) = f(\vec{\alpha})$,
  which completes the proof.
\end{proof}

\begin{proposition}\label{lem:limit-exists}
  Let $\calF$ be a free filter on $\bbN$,
  and $(a_n)_{n=0}^{\infty}$ be a sequence of real numbers.
  If $\lim_{n \to \infty} a_n = \alpha$, then
  \[
    \liminf_{n \to \calF} a_n = \limsup_{n \to \calF} a_n = \alpha  \quad .
  \]
\end{proposition}
\begin{proof}
  Since $\calF$ contains the cofinite filter, by Lemma~\ref{lem:filter-le},
  \[
  \alpha = \liminf_{n \to \infty} a_n
  \le \liminf_{n \to \calF} a_n
  \le \limsup_{n \to \calF} a_n
  \le \limsup_{n \to \infty} a_n = \alpha \quad .
  \]
\end{proof}

\begin{proposition}
  Given an ultrafilter $\calU$ on $I$ and a family of real numbers $(a_i)_{i \in I}$,
  then $\lim_{i \to \calU} a_i$ exists, i.e.,
  $\liminf_{i \to \calU} a_i$ and $\limsup_{i \to \calU} a_i$ coincide.
\end{proposition}
\begin{proof}
  First we show that $\liminf_{i \to \calU} a_i \le \limsup_{i \to \calU} a_i$ holds.
  Let $J, J' \in \calU$.
  By the nonemptiness of $J \cap J'$, we have
  $\inf_{i \in J} a_i \le \inf_{i \in J \cap J'} a_i \le \sup_{i \in J \cap J'} a_i \le \sup_{i \in J'} a_i$.
  Since this inequality holds for any $J$ and $J'$, we have
  $\sup_{J \in \calU} \inf_{i \in J} a_i \le \sup_{J' \in \calU} \inf_{i \in J'} a_i$.

  Now we show the equality.
  (1) Consider $\limsup_{i \in \calU} a_i - \liminf_{i \in \calU} a_i \lt \infty$,
  i.e, $\sup_{i \in J_0} a_i - \inf_{i \in J_0} a_i \lt \infty$ for some $J_0 \in \calU$.
  Let $(E_k)_{k=1}^{m}$ be a division of the interval
  $[\inf_{i \in J_0} a_i, \sup_{i \in J_0} a_i]$ to intervals whose lengths are smaller than $\e$.
  We define $(J_i)_{k=1}^{m}$ by $J_i = \set{i \in J_0}{a_i \in E_i}$.
  Since $\bigcup_{k=1}^{m} J_i = J_0 \in \calU$ and $\calU$ is an ultrafilter,
  there exists $k'$ such that $J_{k'} \in \calU$, and
  by the construction of $J_{k'}$, $\sup_{i \in J_{k'}} a_i - \inf_{i \in J_{k'}} a_i \le \e$ holds.
  Thus we conclude $\limsup_{i \in \calU} a_i - \liminf_{i \in \calU} a_i \le \e$.
  (2) Consider $\limsup_{i \to \calU} a_i = \infty$.
  For $M \in \bbR$,
  we have the division $I = J_0 \cup J_1$ for $J_0 = \set{i \in I}{a_i \lt M}$
  and $J_1 = \set{i \in I}{a_i \ge M}$. Since $\sup_{i \in J_0} a_i \le M \lt \infty$,
  we have $J_0 \not\in \calU$ and then $J_1 \in \calU$.
  Therefore we have $M \le \inf_{i \in J_1} a_i \le \liminf_{i \to \calU} a_i$
  and this inequality holds for any $M \in \bbR$.
  (3) $\liminf_{i \to \calU} a_i = - \infty$ is dual.
\end{proof}

\subsection{Ultralimit of Metric Spaces}
\label{sec:ultralimit-metsp}

\emph{Ultralimit} of metric spaces is introduced in \cite{Kapovich1995}.
It is a metric variant of ultraproduct of first order structures,
and in some sense it is considered as the \emph{limit}
(in the topological sense) of metric spaces.

\begin{lemma}
  Let $\calF$ be a filter on $I$.
  For a family $(X_i, d_i)_{i \in I}$ of metric spaces,
  the function
  $\theta \colon {\prod_{i} X_i} \times {\prod_{i} X_i} \to [0, \infty]$
  defined by
  \[
    \theta((x_i)_i, (y_i)_i) = \limsup_{i \to \calF} d_i(x_i, y_i)
  \]
  is a pseudometric.
\end{lemma}
\begin{proof}
  Let $x = (x_i)_i$, $y = (y_i)_i$ and $z = (z_i)_i$ be sequences of points
  where $x_i, y_i, z_i \in X_i$.
  First, by taking $J = I$, we have
  $\theta(x, x) \le \sup_{i \in I} d_i(x_i, x_i) = 0$.
  Next $\theta(x, y) = \theta(y, x)$ obviously follows from the symmetricity of the definition.
  Finally, for the triangle inequality:
  \begin{align*}
    \theta(x, y) + \theta(y, z)
    &= \limsup_{i \to \calF} d_i(x_i, y_i) + \limsup_{i \to \calF} d_i(y_i, z_i) \\
    &\le \limsup_{i \to \calF} (d_i(x_i, y_i) + d_i(y_i, z_i)) \quad \tag*{(Lemma~\ref{lem:sup-plus})}\\
    &\le \limsup_{i \to \calF} d_i(x_i, z_i) = \theta(x, z) \quad .
  \end{align*}
\end{proof}

\begin{definition}
  \label{def:ulralimit-metsp}
  Let $\calF$ be a filter on $I$.
  For a family $(X_i, d_i)_{i \in I}$ of metric spaces,
  the \emph{reduced limit} of $(X_i, d_i)_{i \in I}$ by $\calF$
  is a metric space $\prod_i^{\calF} (X_i, d_i) = X \quot \theta$ where $X = \prod_{i \in I} X_i$
  and $\theta((x_i)_i, (y_i)_i) = \limsup_{i \to \calF} d_i(x_i, y_i)$.
  It is called \emph{ultralimit} \cite{Kapovich1995} when $\calF$ is an ultrafilter.
\end{definition}

The pointwise limit of metrics can be viewed as an example of ultralimit.

\begin{proposition}
  Let $\calF$ be a free filter on $\bbN$.
  Let $X$ be a set, and $d_\omega$ and $(d_n)_{n \in \omega}$ be metrics on $X$.
  If $d_\omega(x, y) = \lim_{n \to \infty} d_n(x, y)$, then
  $(X, d_\omega)$ is isometric to a subspace of $\prod_n^{\calF} (X, d_n)$
  by $f \colon X \to \prod_n^{\calF} (X, d_n);\, x \mapsto [x]_{n}$.
\end{proposition}
\begin{proof}
  For $x, y \in X$, we have
  \begin{align*}
    d([x]_n, [y]_n)
    &= \liminf_{n \to \calF} d_n(x, y) \tag*{(by definition)} \\
    &= \lim_{n \to \infty} d_n(x, y) \tag*{(Lemma~\ref{lem:limit-exists})} \\
    &= d_{\omega}(x, y) \quad . \tag*{(by definition)}
  \end{align*}
  Therefore $f$ is an embedding.
\end{proof}

At first sight, ultralimit appears to be just a technical generalization of
classical ultraproduct. However it can be understood from a more topological (or metric) point of view
for compact spaces. First we review the Hausdorff distance.

\begin{definition}[e.g.\ \cite{Gromov1999}]
  Let $(X, d)$ be a metric space.
  For $x \in X$ and $A, B \subset X$, we define
  $d(x, A)$ and $d_H^X(A, B)$ as follows.
  \begin{align*}
    d(x, A) &:= \inf_{a \in A} d(x, a)  \\
    d_H^X(A, B) &:= \max \biggl( \sup_{x \in A} d(x, B),\, \sup_{y \in B} d(y, A) \biggr)
  \end{align*}
  This construction defines a metric on the set of closed subsets of $X$,
  which is called the \emph{Hausdorff metric}.
\end{definition}

The Hausdorff metric gives a way to measure a distance \emph{between subsets} of a fixed metric space.
Using this metric, we can define a distance \emph{between metric spaces}
by embedding.

\begin{definition}[\cite{Gromov1999}]
  For compact metric spaces $X$ and $Y$,
  the \emph{Gromov-Hausdorff distance}
  $d_{GH}(X, Y)$ is defined to be the infimum of $d_H^Z(f(X), g(Y))$
  where $Z$ is a metric space and
  $f \colon X \to Z$ and $g \colon Y \to Z$ are embeddings.
  This defines a metric on the set of (isometric classes of) compact metric spaces,
  which is called the \emph{Gromov-Hausdorff metric}.
\end{definition}

The ultralimit of compact metric spaces is indeed a generalization of
the limit in the Gromov-Hausdorff metric.

\begin{proposition}[\cite{Kapovich1995}]
  \label{ultralimit-ghlimit}
  Let $(X_n)_{n=1}^{\infty}$ and $X$ be compact metric spaces
  and $\calF$ be a free ultrafilter on $\bbN$.
  If $(X_n)_{n=1}^{\infty}$ converges to $X$ in the Gromov-Hausdorff metric,
  then $X$ is isometric to $\prod_{n}^\calF X_n$. \qed
\end{proposition}

\begin{remark}
  The converse of Proposition~\ref{ultralimit-ghlimit} does not hold.
  Take two distinct metric spaces $X$ and $Y$,
  and consider the sequence $(X, Y, X, Y, \ldots)$.
  The ultralimit always exists but this sequence
  does not have a limit with respect to the Gromov-Hausdorff metric.
\end{remark}

The ultralimit construction preserves some metric and topological properties of metric spaces;
see \cite{VanDenDries1984, Kapovich1995} for more details and examples.

\begin{proposition}
  Let $\calF$ be a filter on $I$ and $(X_i)_{i \in I}$ be a family of metric spaces.
  If each $X_i$ is compact, then $\prod_{i}^{\calF} X_i$ is also compact.
\end{proposition}
\begin{proof}
  We know that $\prod_{i} X_i$ is compact by Tychonoff's theorem.
  The canonical surjection $\pi \colon \prod_{i} X_i \to \prod_{i}^{\calF} X_i$ is non-expansive,
  then it is also continuous.
  Therefore its image $\prod_{i}^{\calF} X_i$ is also compact.
\end{proof}

\begin{proposition}[e.g.\ \cite{Yaacov2007}]
  \label{thm:uprod-comp-comp}
  Let $\calU$ be an ultrafilter on $I$ and $(X_i)_{i \in I}$ be a family of metric spaces.
  If each $X_i$ is complete, then $\prod_{i}^{\calU} X_i$ is also complete.
\end{proposition}
\begin{proof}
  Let $(x^n)_{n=0}^{\infty}$ be a sequence in $\prod_{i}^\calU X_i$
  where $x^n = [x^n_i]_{i}$, and assume $d(x^n, x^{n+1}) \lt 2^{-n}$ holds for each $n \in \bbN$.
  We show $(x^n)_{n=0}^{\infty}$ converges.
  Let $A_n \subset I$ be the set defined by:
  \begin{align*}
    A_n &= \set{i \in I}{\text{$d(x^{k}_i, x^{k+1}_i) \lt 2^{-k}$ for all $0 \le k \lt n$}} \quad.
  \end{align*}
  Given $i \in I$, define $y_i \in X_i$ as follows: $y_i = x^n_i$
  if $i \in A_n \setminus A_{n+1}$ for some $n$, and otherwise $y_i = \lim_{n \to \infty} x^n_i$.
  We use the completeness of $X_i$ here.
  Then $d(x^n, y) \le 2^{-n+1}$ holds for each $n$, which concludes $\lim_{n \to \infty} x^n = y$.
\end{proof}

In fact, the ultraproduct of a countable family of metric spaces is
automatically complete, even if each metric space is not complete.

\begin{proposition}[\cite{Bridson1999}]\label{thm:cble-uprod-comp}
  Let $\calU$ be a free ultrafilter on $\bbN$
  and $(X_i)_{i \in I}$ be a family of (not necessarily complete) metric spaces.
  Then $\prod_{i}^\calU X_i$ is complete.
\end{proposition}
\begin{proof}
  Define $A_n$ as in Proposition~\ref{thm:uprod-comp-comp}
  and $B_n = A_n \setminus \{ 0, \ldots, n\}$.
  We know $B_n \in \calU$ since $\calU$ is free,
  and $\bigcap_{n=0}^{\infty} B_n = \emptyset$ holds by construction.
  We define $y_i = x^n_i$ where $n$ satisfies $i \in A_n \setminus A_{n+1}$,
  and then we have $\lim_{n \to \infty} x^n = y$ as Proposition~\ref{thm:uprod-comp-comp}.
\end{proof}

In the following corollary, you can find an immediate application of Proposition~\ref{thm:cble-uprod-comp}.
We can prove the existence of the completion of a metric space.

\begin{corollary}\label{compl-via-ultraprod-metsp}
  Given a metric space $X$, its completion exists.
\end{corollary}
\begin{proof}
  Let $\calU$ be some free ultrafilter on $\bbN$.
  Consider the ultrapower $\prod_{n}^{\calU} X$
  and a canonical map $\iota \colon X \to \prod_{n}^{\calU} X$.
  Since $\iota$ is isometric and $\prod_{n}^{\calU} X$ is complete,
  the closure of $\iota(X)$ is the completion of $X$.
\end{proof}

Note that we use the existence of real nubmers in Theorem~\ref{thm:cble-uprod-comp},
so Corollary~\ref{compl-via-ultraprod-metsp} cannot entirely replace
the usual construction of completions by Cauchy sequences.

\section{Metric and Quantitative Algebra}
\label{ch:malg-qalg}

In this section, we introduce the notion of metric algebra and quantitative algebra.
We also introduce some elementary constructions of metric algebras
such as subalgebra, product and quotient.
These constructions are used in the metric version of the variety theorem.

\subsection{Metric and Quantitative Algebra}

By combining metric structures and $\Sigma$-algebraic structures,
we acquire the definitions of metric algebra. They go as follows.

\begin{definition}[\cite{Weaver1995}] \mbox{}
  \begin{itemize}
    \item A \emph{metric algebra} is a tuple $A = (A, d, (\sigma^A)_{\sigma \in \Sigma})$
    where $(A, d)$ is a metric space and $(A, (\sigma^A)_{\sigma \in \Sigma})$
    is a $\Sigma$-algebra.
    We denote the class of metric algebras by $\mathcal{M}$.

    \item A map $f \colon A \to B$ between metric algebras is called
    a \emph{homomorphism} if $f$ is $\Sigma$-homomorphic and non-expansive.

    \item A \emph{subalgebra} of a metric algebra $A$
    is a subalgebra (as $\Sigma$-algebra) equipped with the induced metric.
    An \emph{embedding} is an isometric homomorphism.

    \item The \emph{product} of metric algebras is
    the product (as $\Sigma$-algebras) equipped with the supremum metric.

    \item A \emph{quotient} of a metric algebra $A$ is a pair $(B, \pi)$
    where $B$ is a metric algebra and $\pi \colon A \epi B$ is a surjective homomorphism.
  \end{itemize}
\end{definition}

The definition of metric algebra says nothing about the relationship
between its metric structure and its algebraic structure.
One natural choice is to require their operations to be non-expansive,
which leads us to quantitative algebra.

\begin{definition}[\cite{Mardare2016}]
  A metric algebra $A$ is a \emph{quantitative algebra} if each
  $\sigma^{A} \colon A^n \to A$ is non-expansive for each $\sigma \in \Sigma_n$,
  where $A^n$ is equipped with the supremum metric.
  We denote the class of quantitative algebras by $\mathcal{Q}$.
\end{definition}

The non-expansiveness requirement for operations
is categorically natural since it says that a quantitative algebra
is an algebra in the category of metric spaces and non-expansive maps
in the sense of Lawvere theory.
However this formulation does not allow \emph{normed vector spaces},
since the scalar multiplications are not non-expansive.
More extremely, even $(\bbR, +)$ is not quantitative
since $d(0+0, 1+1) \gt \max(d(0, 1), d(0, 1))$.
Thus basically we try to build our theory for general metric algebras.

Instead of respectively discussing the variety theorems for metric algebras and quantitative algebras,
we show the variety theorem \emph{relative to} a given class $\cat{K}$.
In that case, $\cat{K}$ is well-behaved when it is a \emph{prevariety}.
For example, $\cat{Q}$ and $\cat{M}$ are prevarieties.

\begin{definition}[\cite{Weaver1995}]
  A class of metric algebras is called a \emph{prevariety}
  if it is closed under subalgebras and products.
\end{definition}

\newpara

In \cite{Mardare2017}, Mardare et al.\ introduce the notion of \emph{$\kappa$-reflexive} homomorphism
for a cardinal $\kappa \le \aleph_1$,
and give the characterization theorem of \emph{$\kappa$-varieties} by $\kappa$-reflexive quotients.

\begin{notation}
  For a set $B$ and a cardinal $\kappa$, we write $A \subset_{\kappa} B$ when $A$ is a subset
  whose cardinality is smaller than $\kappa$.
  For example, $A \subset_{\aleph_0} B$ is a finite subset
  and $A \subset_{\aleph_1} B$ is an at most countable subset.
\end{notation}

\begin{definition}[\cite{Mardare2017}]
  \label{def:k-refl-homom}
  A surjective homomorphism $p \colon A \to B$ between metric algebras
  is $\kappa$-reflexive if, for any subset $B' \subset_\kappa B$,
  there exists a subset $A' \subset_\kappa A$
  such that $p$ restricts to a bijective isometry $p|_{A'} \colon A' \to B'$.
\end{definition}

We also use a variant of $\kappa$-reflexive homomorphism in our variety theorem,
but our notion is \emph{unbounded}; we do not impose any size condition.

\begin{definition}
  \label{def:refl-homom}
  A surjective homomorphism $p \colon A \to B$ between metric algebras is \emph{reflexive}
  if there exists a subset $A' \subset A$ such that
  $p |_{A'} \colon A' \to B$ is a bijective isometry.
  Equivalently $p$ is reflexive if and only if there exists
  an embedding $s \colon B \to A$ \emph{as metric space} such that $p \circ s = \id_B$.
  Note that $s$ is not required to be homomorphic.
\end{definition}

In the rest of this paper, when we say ``a class $\cat{K}$ of metric algebras'',
we implicitly assume that $\cat{K}$ is closed under isomorphisms;
it is a natural assumption since we are interested in
properties of metric algebras, and they must be preserved by isomorphisms
between metric algebras.

\subsection{Congruential Pseudometric}
\label{sec:cong-quot}

The notion of quotient seems to be external and difficult to deal with.
For example, it is not trivial to see that the class of quotients of a metric algebra
(up to isomorphism) turns out to be a small set.

In the case of classical universal algebra,
there is a bijective correspondence between quotient algebras and \emph{congruences},
which enables us to treat quotients internally and concretely.
To extend this correspondence to the metric case,
we are led to the notion of \emph{congruential pseudometric} instead of
the usual congruence in classical universal algebra.
The idea of using pseudometrics as the metric version of congruences
also appears in \cite{Mardare2017}.

\begin{definition}\label{def:cong}
  A \emph{congruential pseudometric} of a metric algebra $A$ is a pseudometric $\theta$ on $A$
  such that $\theta(x, y) \le d^A(x, y)$ holds for each $x, y \in A$
  and the equivalence relation $\theta(x, y) = 0$ is a congruence as $\Sigma$-algebra.
  We think that the set of congruential pseudometrics is ordered by the \emph{reversed pointwise order}:
  for $\theta_1$ and $\theta_2$, we say $\theta_1 \cle \theta_2$
  when $\theta_1(a, b) \ge \theta_2(a, b)$ holds for any $a, b \in A$.

  Given a congruential pseudometric $\theta$, the metric space $A \quot \theta$
  is viewed as a metric algebra by the algebra structure defined by
  $\sigma([a_1], \ldots, [a_n]) = [\sigma(a_1, \ldots, a_n)]$
  and equipped with a canonical homomorphic projection $\pi \colon A \to A \quot \theta$.
\end{definition}

We adopt the reversed pointwise order for the consistency with the classical case.
In the classical case, the set of congruences are ordered by inclusion, and
for two congruences $\theta_1 \subset \theta_2$ on $A$, we have a canonical
surjective homomorphism $\pi \colon A \quot \theta_1 \onto A \quot \theta_2$.
In the metric case, for the congruential pseudometrics $\theta_1 \cle \theta_2$,
their metric identifications satisfy $\sim_{\theta_1} \subset \sim_{\theta_2}$,
and we have a surjective homomorphism $\pi \colon A \quot \theta_1 \onto A \quot \theta_2$
between metric algebras.

As in classical universal algebra, we can prove the first isomorphism theorem
and  a bijective correspondence between quotients and congruences.
The other isomorphism theorems are presented in Section~\ref{ch:congruence}.

\begin{definition}
  Given a homomorphism $f \colon A \to B$ between metric algebras,
  \begin{itemize}
    \item The \emph{image of $f$} is a subalgebra of $B$ defined by $\im(f) = \set{f(a)}{a \in A}$.
    \item The \emph{kernel} of $f$ is a congruential pseudometric on $A$ that is defined by $\ker(f)(a, b) = d(f(a), f(b))$.
  \end{itemize}
\end{definition}

\begin{proposition}[First Isomorphism Theorem]
  \label{first-isom-thm}
  Let $f \colon A \to B$ be a homomorphism between metric algebras,
  and $\theta$ be a congruence on $A$.
  If $\theta \cle \ker(f)$,
  there exists a unique homomorphism $\bar{f} \colon A \quot \theta \to B$
  such that $\bar{f}([a]) = f(a)$ for all $a \in A$.
  Moreover if $\theta = \ker(f)$,
  then $\bar{f}$ is an isometry and
  the induced map $\bar{f} \colon A \quot \ker(f) \to \im(f)$
  is an isomorphism.
\end{proposition}
\begin{proof}
  First we show $\bar{f}([a]) = f(a)$ is well-defined.
  Assume $[a] = [b]$, i.e., $\theta(a, b) = 0$.
  Since $\theta \cle \ker(f)$, we have $d(f(a), f(b)) \le \theta(a, b) = 0$
  hence $f(a) = f(b)$.
  Next we show $\bar{f}$ is non-expansive; for $a, b \in A$,
  we have $d(\bar{f}([a]), \bar{f}([b])) = d(f(a), f(b))
  \le \theta(a, b) = d([a], [b])$.
  In the case $\theta = \ker(f)$, we also have
  $d(f(a), f(b)) = \theta(a, b)$ and then $d(\bar{f}([a]), \bar{f}([b])) = d([a], [b])$.
\end{proof}

\begin{corollary}
  For a metric algebra $A$, quotients of $A$ bijectively correspond with
  congruences on $A$ by $\theta \mapsto (A \quot \theta, \pi)$
  and $(B, p) \mapsto \ker(p)$. \qed
\end{corollary}

\begin{proposition}
  For a metric algebra $A$ and a congruence $\theta$ on $A$,
  there is a lattice isomorphism between
  $\Con (A \quot \theta)$ and
  $(\Con A)_{\ge \theta} = \set{\rho \in \Con A}{\rho \ge \theta}$.
\end{proposition}
\begin{proof}
  Given $\rho \in (\Con A)_{\ge \theta}$,
  we define a pseudometric $\bar{\rho}$ on $A \quot \theta$
  by $\bar{\rho}([a], [b]) = \rho(a, b)$. It is well-defined:
  if $\theta(a, a') = \theta(b, b') = 0$, then we have
  $\rho(a, a') = \rho(b, b') = 0$ by $\rho \le \theta$.
  Therefore $\rho(a, b) = \rho(a', b')$ by the triangle inequality.
  Conversely, given $\rho' \in \Con (A \quot \theta)$,
  we have a congruence on $A$ by pulling back $\rho'$
  along the canonical projection $[\blank] \colon A \to A \quot \theta$.

  It is easy to see that they are inverse and order-preserving.
\end{proof}

\subsection{Ultraproduct of Metric Algebras}
\label{sec:ultraprod}

As in classical first order logic, we want to define the reduced product
and the ultraproduct of a family of metric algebras. However there is a difficulty;
the pseudometric defined in Definition~\ref{def:ulralimit-metsp} is not necessarily
a congruential pseudometric, i.e., the relation $\theta(x, y) = 0$ is not preserved by operations.
For this reason, we think of ultraproduct as a partial operation, following \cite{Weaver1995}.

\begin{definition}[\cite{Weaver1995}]
  \label{def:reduced-prod-malg}
  Let $\calF$ be a filter on a set $I$.
  For a family $(A_i)_{i \in I}$ of metric algebras,
  the \emph{reduced product of $(A_i)_{i \in I}$ by $\calF$} \emph{exists}
  when the pseudometric $\theta(x, y) = \limsup_{i \to \calF} d^{A_i}(x, y)$
  on $\prod_{i} A_i$ is congruential.
  When it exists, it is defined by $\prod_{i}^{\calF} A_i = (\prod_{i} A_i) \quot \theta$.
  We denote the equivalence class of $(x_i)_i$ by $[x_i]_i$.

  If $\calF$ is an ultrafilter, it is called an \emph{ultraproduct}.
  Moreover when $A_i = A$ for each $i \in I$, it is called an \emph{ultrapower of $A$}.

  We say that a class $\cat{K}$ of metric algebras is closed under reduced products if,
  for any nonempty set $I$, any filter $\calF$ on $I$, and any family $(A_i)_{i \in I}$ of metric algebras
  in $\cat{K}$, the reduced product of $(A_i)_{i \in I}$
  by $\calF$ exists and belongs to $\cat{K}$.
  We define the closedness under ultraproducts in the same way.
  Note that we require the existence.
\end{definition}

We have no general method to judge whether the ultraproduct exists or not,
but there is a convenient sufficient condition.

\begin{proposition}[\cite{Weaver1995}]
  In Definition~\ref{def:reduced-prod-malg},
  the pseudometric $\theta$ is congruential
  if each $\Sigma$-operation is uniformly equicontinuous:
  for any $\sigma \in \Sigma_n$ and $\e \gt 0$,
  there is $\delta \gt 0$ such that
  for any $i \in I$ and $\vec{a}, \vec{b} \in A_i^n$
  with $d(\vec{a}, \vec{b}) \le \delta$, we have
  $d(\sigma(\vec{a}), \sigma(\vec{b})) \le \e$
  (Here we define $d(\vec{a}, \vec{b}) = \max_{k} d(a_k, b_k)$).

  In particular, when $A_i = A$ for all $i \in I$ and each $\sigma^{A}$ is uniformly continuous,
  then the ultrapower of $A$ by $\calF$ exists.
\end{proposition}

\begin{corollary}
  The reduced product of $(A_i)_{i \in I}$ exists in the following cases.
  \begin{itemize}
    \item When $(A_i)_{i \in I}$ is a family of quantitative algebras.
    \item When $(A_i)_{i \in I}$ is a family of normed vector spaces.
  \end{itemize}
\end{corollary}

As in the case of metric spaces in Corollary~\ref{compl-via-ultraprod-metsp},
we can construct the completion of a metric algebra via ultraproduct.

\begin{proposition}[\cite{Weaver1995}]\label{compl-via-ultraprod-malg}
  If a class $\cat{K}$ of metric algebras is closed under ultraproducts and subalgebras,
  then $\cat{K}$ is also closed under completions.
\end{proposition}
\begin{proof}
  The same construction as Corollary~\ref{compl-via-ultraprod-metsp} works.
  Note that the ultrapower of $A$ exists since we assume so.
\end{proof}

We can think, in some sense, that ultraproduct defines a ``topology''
on the class of metric algebras\footnote{If we appropriately restrict the size,
this construction gives rise to a topological space.},
as the Gromov-Hausdorff metric defines a metric on the set of compact metric spaces.

\begin{definition}
  \label{def:continuity}
  A class $\cat{K}$ of metric algebras is called \emph{continuous}
  if it is closed under taking ultraproducts.
\end{definition}

\begin{proposition}
  \label{ultraprod-topology}
  If $\{ \calC_\lambda \}_{\lambda \in \Lambda}$, $\calC_1$ and $\calC_2$ are continuous classes
  of metric algebras, then $\bigcap_{\lambda \in \Lambda} \calC_\lambda$
  and $\calC_1 \cup \calC_2$ are also continuous.
\end{proposition}
\begin{proof}
  The continuity of $\bigcap_{\lambda \in \Lambda} \calC_\lambda$ is obvious.

  Let $(A_i)_{i \in I}$ be a family of metric algebras
  where $A_i \in \calC_1 \cup \calC_2$, and $\calU$ be an ultrafilter on $I$.
  Let us define $J_1$ and $J_2$ by $J_k = \set{i \in I}{A_i \in \calC_k}$.
  Since $J_1 \cup J_2 = I$ and $\calU$ is an ultrafilter,
  either $J_1 \in \calU$ or $J_2 \in \calU$ holds;
  we can assume $J_1 \in \calU$ without loss of generality.
  Then $\prod_i^{\calU} A_i \isom \prod_i^{\calU |_{J_1}} A_i \in \calC_1 \subset \calC_1 \cup \calC_2$.
\end{proof}

In the light of Proposition~\ref{ultraprod-topology},
it might be more natural to adopt the adjective \emph{closed} rather than \emph{continuous}.
However the use of \emph{closed} seems confusing
since we also use it for the closedness under algebraic operations,
therefore we prefer the word \emph{continuous}.
As we will see in Section~\ref{ch:logical-metric}, this terminology is consistent
with \emph{continuous quasivariety} defined in \cite{Khudyakov2003}.

\subsection{Closure Operator}
\label{sec:closure-op}

It is sometimes convenient to view a construction of metric algebras
as an operator on classes of metric algebras. See \cite{Gorbunov1998} for the classical case, and \cite{Mardare2017} for the quantitative case.

\begin{definition}[cf.\ \cite{Mardare2017}]
  We define the class operators $\bbH$, $\bbH_r$, $\bbS$, $\bbP$ and $\bbP_U$ as follows.
  \begin{align*}
    \bbH(\cat{K}) &= \set{A \in \cat{M}}{\text{$A$ is a quotient of some $B \in \cat{K}$}} \\
    \bbH_r(\cat{K}) &= \set{A \in \cat{M}}{\text{$A$ is a reflexive quotient of some $B \in \cat{K}$}} \\
    \bbS(\cat{K}) &= \set{A \in \cat{M}}{\text{$A$ is a subalgebra of some $B \in \cat{K}$}} \\
    \bbP(\cat{K}) &= \set{A \in \cat{M}}{\text{$A$ is the product of some $(B_i)_{i} \in \cat{K}$}} \\
    \bbP_U(\cat{K}) &= \set{A \in \cat{M}}{\text{$A$ is an ultraproduct of some $(B_i)_{i} \in \cat{K}$}}
  \end{align*}
  For class operators $\bbA$ and $\bbB$, we denote their composition by
  $\bbB \bbA (\cat{K}) = \bbB (\bbA (\cat{K}))$,
  and write $\bbA \subset \bbB$ when $\bbA(\cat{K}) \subset \bbB(\cat{K})$ holds for any class $\cat{K}$.
\end{definition}

\begin{proposition}[cf.\ \cite{Mardare2017}] \mbox{}
  \label{classop-comp}
  \begin{itemize}
    \item $\bbS \bbH \subset \bbH \bbS$ and $\bbS \bbH_r \subset \bbH_r \bbS$.
    \item $\bbP \bbH \subset \bbH \bbP$ and $\bbP \bbH_r \subset \bbH_r \bbP$.
    \item $\bbP_U \bbH \subset \bbH \bbP_U$ and $\bbP_U \bbH_r \subset \bbH_r \bbP_U$.
  \end{itemize}
\end{proposition}
\begin{proof}
  The proof is almost analogous to the classical case \cite{Gorbunov1998};
  we only give a proof for $\bbS \bbH_r \subset \bbH_r \bbS$.
  Let $\cat{K}$ be a class of metric algebras.
  Assume $A \in \cat{K}$ and $B' \in \bbS \bbH_r (\cat{K})$, that is,
  there exists a reflexive homomorphism $p \colon A \epi B$
  and $B' \subset B$ is a subalgebra.
  Let $s \colon B \mono A$ be a metric embedding such that $p \circ s = \id_B$
  and let us define $A' = \set{a \in A}{p(a) \in B'}$.
  Then $A'$ is a subalgebra of $A$ and $s$ restricts to a map $s |_{B'} \colon B' \to A'$.
  Therefore $B'$ is a reflexive quotient of $A'$, thus $B' \in \bbH_r \bbS (\cat{K})$.
\end{proof}

\section{Congruence Lattice on Metric Algebras}
\label{ch:congruence}

Congruence not only gives a concrete description of quotient
but is a fundamental tool in universal algebra.
We can characterize various constructions of $\Sigma$-algebra by congruence,
and use the congruence theory in the proof of the variety theorem.

As we saw in the previous section, the notion of congruence is generalized
to congruential pseudometric in the theory of metric algebra. In this section,
we give the metric counterpart of the congruence theory in classical universal algebra.

\subsection{Isomorphism Theorem}
\label{sec:isothm}

We showed the metric version of the first isomorphism theorem in Proposition~\ref{first-isom-thm}.
In this section we prove the rest of the isomorphism theorems.

\begin{definition}
  Let $A$ be a metric algebra and $\theta$ be a congruential pseudometric on $A$.
  \begin{itemize}
    \item For a subalgebra $B$ of $A$, \emph{the restriction of $\theta$ to $B$}
      is defined by the usual restriction of pseudometric, which we denote by $\theta_B$.
    \item For a subset $S \subset A$,
      we define
      $S^\theta = \set{a \in A}{\exists s \in S, \, d(s, a) = 0}$.
  \end{itemize}
\end{definition}

\begin{theorem}
  In the situation above, if $S$ is a subalgebra of $A$, so is $S^{\theta}$.
\end{theorem}
\begin{proof}
  Let $\sigma \in \Sigma$ be an $n$-ary operation in $\Sigma$.
  Suppose $a_1, \ldots, a_n \in S^{\theta}$. By the definition of $S^{\theta}$,
  there exists $s_1, \ldots, s_n \in S$ such that
  $s_i \sim_{\theta} a_i$ for $i = 1, \ldots, n$.
  Since the relation $\sim_{\theta}$ is preserved by $\sigma^A$, we also have
  $\sigma(s_1, \ldots, s_n) \sim_{\theta} \sigma(a_1, \ldots, a_n)$.
  Since $S$ is a subalgebra, we have $\sigma(s_1, \ldots, s_n) \in S$
  and then we conclude $\sigma(a_1, \ldots, a_n) \in S^{\theta}$.
\end{proof}

\begin{theorem}[Second and Third Isomorphism Theorem] \mbox{}
  \begin{enumerate}
    \item Given a metric algebra $A$, a subalgebra $B$ of $A$ and
      a congruence $\theta$ on $A$, we have a canonical isomorphism
      $B^\theta \quot \theta_{B^\theta} \isom B \quot \theta_B$.
    \item Given a metric algebra $A$ and congruences $\rho$, $\theta$ on $A$ with $\rho \cle \theta$,
      we have a canonical isomorphism $A \quot \rho \isom (A \quot \theta) \quot (\rho \quot \theta)$.
  \end{enumerate}
\end{theorem}
\begin{proof}
  (1) Let $f \colon B \to B^\theta$ be an inclusion map.
    Since $\theta_B$ is a restriction of $\theta_{B^\theta}$,
    we have $\theta_B(x, y) = \theta_{B^\theta}(x, y)$ for each $x, y \in B$.
    Then $f$ induces an embedding $\bar{f} \colon B \quot \theta_B \to B^\theta \quot \theta_{B^\theta}$.
    It is surjective; for $x \in B^\theta$, there exists $y \in B$ such that
    $\theta_{B^\theta}(x, y) = \theta(x, y) = 0$ by the definition of $B^\theta$, therefore
    $\bar{f}([y]) = [x]$.

  (2) Let $\pi \colon A \to A \quot \theta$ be the natural projection.
  It is easy to see that $\pi$ induces an isomorphism
  $\bar{\pi} \colon A \quot \rho \isom (A \quot \theta) \quot (\rho \quot \theta)$
  as (1).
\end{proof}

\subsection{Congruence Lattice}
\label{sec:cong-lat}

In classical universal algebra, it is sometimes convenient to consider
the poset of congruences rather than a congruence (see e.g.\ \cite{Burris1981}).

In this section, we show that the poset of congruential pseudometrics is a complete lattice
as in the classical case. Thus we can take arbitrary join and meet
of congruential pseudometrics.

\begin{definition}
  Let $\cat{K}$ be a class of metric algebras.
  A congruence $\theta$ on $A$ is \emph{$\cat{K}$-congruential}
  if $A \quot \theta$ belongs to $\cat{K}$.
  We denote by $\Con(A)$ the set of congruences on $A$,
  and by $\Con_{\cat{K}}(A)$ the set of $\cat{K}$-congruential pseudometrics on $A$.
\end{definition}

\begin{definition}[\cite{Weaver1995}]
  \label{def:subdir-prod}
  Let $(A_i)_{i \in I}$ be a family of metric algebras.

  A \emph{subdirect product of $(A_i)_{i \in I}$} is a subalgebra $A$
  of the product $\prod_{i \in I} A_i$ where each projection map $\pi_i \colon A \to A_i$ is surjective.

  A homomorphism $f \colon A \to \prod_{i \in I} A_i$ between metric algebras is a \emph{subdirect embedding}
  if $f$ is an embedding and $f(A)$ is a subdirect product of $(A_i)_{i \in I}$,
  that is, each component $f_i \colon A \to A_i$ is surjective.
\end{definition}

\begin{lemma}
  Let $A = (A, d)$ be a metric algebra, $(\theta_i)_{i \in I}$ be a family of congruences on $A$
  and $f \colon A \to \prod_{i \in I} A \quot \theta_i$ be the product of their projections.
  Then its kernel
  is presented by  $\ker(f)(a, b) = \sup_{i \in I} \theta_i (a, b)$ for $a, b \in A$.
  Moreover, the induced map $\bar{f} \colon A \quot \ker(f) \to \prod_{i \in I} A \quot \theta_i$
  is a subdirect embedding.
\end{lemma}
\begin{proof}
  Since $\prod_{i \in I} A \quot \theta_i$ is endowed with the supremum metric,
  then we have $\ker(f)(a, b) := d(f(a), f(b)) = \sup_{i \in I} \theta_i(a, b)$.
  The rest of the theorem follows from Proposition~\ref{first-isom-thm}.
\end{proof}

\begin{corollary}\label{cor:cong-complat}
  If $(\theta_i)_{i \in I}$ is a family of congruences on $A$,
  then $\theta(a, b) = \sup_{i \in I} \theta_i (a, b)$ is also a congruence on $A$.
  If $\cat{K}$ is closed under subdirect products
  and each $\theta_i$ is $\cat{K}$-congruential,
  then $\theta$ is also $\cat{K}$-congruential.
\end{corollary}

Therefore $\Con(A)$ is a complete lattice,
and if $\cat{K}$ is closed under subdirect products, $\Con_{\cat{K}}(A)$
is also a complete lattice.
We denote the meet and join of $(\theta_i)_{i \in I}$ in $\Con(A)$
by $\bigcurlywedge_{i \in I} \theta_i$ and $\bigcjoin_{i \in I} \theta_i$ respectively.
Recall that we adopt the reversed pointwise order for congruences,
so Corollary~\ref{cor:cong-complat} means
that the meet of congruences in $\Con(A)$ and $\Con_{\cat{K}}(A)$ is their pointwise supremum.

In general, it is difficult to give a concrete description of the join of congruences,
but it can be done for some cases.
For example, the assumption of the following theorem is satisfied
if $A \quot \theta_i$ is a quantitative algebra, or if $A \quot \theta_i$ is a normed vector space.

\begin{theorem}\label{thm:cong-join}
  Let $(\theta_i)_{i \in I}$ be congruences on $A$, and assume the following condition:
  for each $\sigma \in \Sigma_n$, there exists a positive real number $K_{\sigma}$ such that
  for any $i \in I$ and $\vec{a}, \vec{b} \in A^n$
  we have $\theta_i(\sigma(\vec{a}), \sigma(\vec{b})) \le K_{\sigma} \theta_i(\vec{a}, \vec{b})$.
  Then we have:
  \[
    \bigl( \bigcjoin_{i \in I} \theta_i \bigr) (a, b)
    = \inf_{\substack{n \ge 0, \, c_1, \ldots, c_n \in A \\ i_0, i_1, \ldots, i_n \in I}}
      \bigl( \theta_{i_0}(a, c_1) + \theta_{i_1}(c_1, c_2) + \cdots + \theta_{i_n}(c_n, b) \bigr)
      \quad .
  \]
  Moreover
  $(\bigcjoin_{i \in I} \theta_i)(\sigma(\vec{a}), \sigma(\vec{b})) \le K_{\sigma} (\bigcjoin_{i \in I} \theta_i)(\vec{a}, \vec{b})$
  holds for $\vec{a}, \vec{b} \in A^n$.
\end{theorem}
\begin{proof}
  Let $\theta = \bigcjoin_{i \in I} \theta_i$ and $\rho(a, b)$ be the right hand side.

  ($\le$) Since $\theta_i \cle \theta$ for all $i \in I$, we have:
  \begin{align*}
    \theta_{i_0}(a, c_1) + \theta_{i_1}(c_1, c_2) &+ \cdots + \theta_{i_n}(c_n, b) \\
    &\ge \theta(a, c_1) + \theta(c_1, c_2) + \cdots + \theta(c_n, b) \\
    &\ge \theta(a, b) \quad.
  \end{align*}
  Taking the infimum, we have
  $\theta(a, b) \le \rho(a, b)$

  ($\ge$) Since $\rho(a, b) \le \theta_i(a, b)$ for each $i \in I$ and $a, b \in A$,
  it is sufficient to show that $\rho$ is congruential.
  It is easy to see that $\rho$ is a pseudometric,
  so we only have to show that each $\sigma \in \Sigma$ preserves the metric identification $\rho(x, y) = 0$.
  We only prove the case $\abs{\sigma} = 1$ for the simplicity;
  the other cases are very similar.

  Suppose $\rho(a, b) = 0$, that is, for $\e \gt 0$, there exists
  $c_1, \ldots, c_n \in A$ and $i_0, \ldots, i_n \in I$ such that
  $\theta_{i_0}(a, c_1) + \cdots + \theta_{i_n}(c_n, b) \le \e$.
  Since $\theta_i(\sigma(x), \sigma(y)) \le K_{\sigma} \theta_i(x, y)$, we also have
  $\rho(\sigma(a), \sigma(b)) \le \theta_{i_0}(f(a), f(c_1)) + \cdots + \theta_{i_n}(f(c_n), f(b)) \le K_{\sigma} \e$.
  Therefore $\rho(\sigma(a), \sigma(b)) = 0$ by letting $\e \to 0$, which completes the proof.
\end{proof}

\begin{corollary}
  If $\theta_i$ is $\cat{Q}$-congruential for $i \in I$,
  then $\bigcjoin_{i} \theta_i$ is $\cat{Q}$-congruential.
\end{corollary}

\subsection{Permutable Congruences}
\label{sec:percong}

In the classical case, products are characterized by permutable congruences;
this generalizes the characterization theorem of product of groups via normal subgroups,
and that of product of commutative rings by ideals (see \cite{Burris1981}).

In this section, we prove the metric version of this characterization theorem;
in our formulation, completeness is crucial to prove the theorem.

\begin{definition}
  Let $A$ be a metric algebra.
  For congruences $\theta_1$ and $\theta_2$ on $A$,
  a function $\theta_1 \circ \theta_2 \colon A \times A \to \bbRp$ is defined by
  \[
    (\theta_1 \circ \theta_2)(a, b) = \inf_{c \in A} (\theta_1(a, c) + \theta_2(c, b)) \quad.
  \]
  The congruences $\theta_1$ and $\theta_2$ are \emph{permutable}
  if $\theta_1 \circ \theta_2 = \theta_2 \circ \theta_1$ holds.
\end{definition}

\begin{lemma}\label{lem:cong-prod}
  Let $A$ be a metric algebra and $\theta_1$, $\theta_2$ be congruences on $A$.
  Then the following propositions hold:
  \begin{enumerate}
    \item $\theta_i \cle \theta_1 \circ \theta_2$ for each $i = 1, 2$.
    \item $(\theta_1 \circ \theta_2) (a, a) = 0$ for each $a \in A$.
    \item $\theta_1 \circ \theta_2 \cle \theta_1 \curlyvee \theta_2$.
  \end{enumerate}
\end{lemma}
\begin{proof}
  (1) For $a, b \in A$, we have
  $\theta_1 \circ \theta_2 (a, b) \le \theta_1 (a, b) + \theta_2 (b, b) = \theta_1 (a, b)$.
  The case $i = 2$ is exactly the same.

  (2) It directly follows from (1) and $\theta_1(a, a) = 0$.

  (3) Let $\theta = \theta_1 \curlyvee \theta_2$.
  For any $a, b, c \in A$,
  we have $\theta_1(a, c) + \theta_2(c, b) \ge \theta(a, c) + \theta(c, b) \ge \theta(a, b)$.
  Taking the infimum over $c$, we have $(\theta_1 \circ \theta_2) (a, b) \ge \theta(a, b)$.
\end{proof}

\begin{theorem}
  For congruences $\theta_1$ and $\theta_2$, the followings are equivalent:
  \begin{enumerate}
    \item $\theta_1 \circ \theta_2 = \theta_2 \circ \theta_1$, i.e., they are permutable.
    \item $\theta_1 \circ \theta_2 = \theta_1 \curlyvee \theta_2$.
    \item $\theta_2 \circ \theta_1 \cle \theta_1 \circ \theta_2$.
  \end{enumerate}
\end{theorem}
\begin{proof}
  $(1 \Rightarrow 2)$ By (1) and (3) of Lemma~\ref{lem:cong-prod}, we only have to show
  $\theta_1 \circ \theta_2$ is a congruence. By (2) of Lemma~\ref{lem:cong-prod},
  we have $(\theta_1 \circ \theta_2)(a, a) = 0$.
  For $a, b \in A$, we have $(\theta_1 \circ \theta_2)(a, b) = (\theta_2 \circ \theta_1)(b, a)$,
  and by the permutability, it is equal to $(\theta_1 \circ \theta_2)(b, a)$.
  It remains to prove the triangle inequality.
  For $a, b, c \in A$,
  \[
    (\theta_1 \circ \theta_2)(a, b) + (\theta_1 \circ \theta_2)(b, c)
    = \inf_{d, e \in A} ( \theta_1(a, d) + \theta_2(d, b) + \theta_1(b, e) + \theta_2(e, c) )
  \]
  by definition. Let us fix $\e \gt 0$. Since $\theta_1 \circ \theta_2 = \theta_2 \circ \theta_1$,
  there exists some $g \in A$ such that
  $\theta_1(d, g) + \theta_2(g, e) \le \theta_2(d, b) + \theta_1(b, e) + \e$.
  Then
  \begin{align*}
    \theta_1(a, d) + \theta_2(d, b) &+ \theta_1(b, e) + \theta_2(e, c) + \e \\
    &\ge \theta_1(a, d) + \theta_1(d, g) + \theta_2(g, e) + \theta_2(e, c) \\
    &\ge \theta_1(a, g) + \theta_2(g, c) \\
    &\ge (\theta_1 \circ \theta_2) (a, c) \quad .
  \end{align*}
  Letting $\e \to 0$, we have
  $\theta_1(a, d) + \theta_2(d, b) + \theta_1(b, e) + \theta_2(e, c) \ge (\theta_1 \circ \theta_2) (a, c)$.
  By taking the infimum over $d$ and $e$, the proof is complete.

  $(2 \Rightarrow 3)$ By (3) of Lemma~\ref{lem:cong-prod},
  we have $\theta_2 \circ \theta_1 \cle \theta_2 \cjoin \theta_1 = \theta_1 \cjoin \theta_2 = \theta_1 \circ \theta_2$.

  $(3 \Rightarrow 1)$ It suffices to show $\theta_2 \circ \theta_1 \cge \theta_1 \circ \theta_2$.
  For $a, b \in A$, we have
  $(\theta_2 \circ \theta_1) (a, b) = (\theta_1 \circ \theta_2)(b, a) \le (\theta_2 \circ \theta_1) (b, a)s
  = (\theta_1 \circ \theta_2)(a, b)$, which concludes the proof.
\end{proof}

\begin{lemma}
  \label{lem:quot-perm}
  Let $A$ be a metric algebra and $\theta, \theta_1, \theta_2$ be congruences
  on $A$ satisfying $\theta \cle \theta_i$ for $i = 1, 2$.
  Then $\theta \cle \theta_1 \circ \theta_2$
  and $(\theta_1 \circ \theta_2) \quot \theta = (\theta_1 \quot \theta) \circ (\theta_2 \quot \theta)$.

  In particular, if $\theta_1$ and $\theta_2$ are permutable,
  $\theta_1 \quot \theta$ and $\theta_2 \quot \theta$ are also permutable.
\end{lemma}
\begin{proof}
  For any $a, b \in A$,
  \begin{align*}
      (\theta_1 \circ \theta_2) (a, b) &= \inf_{c \in A} (\theta_1(a, c) + \theta_2(c, b)) \\
      &\le \inf_{c \in A} (\theta(a, c) + \theta(c, b)) \\
      &\le \theta(a, b) \quad .
  \end{align*}
  Therefore $\theta_1 \circ \theta_2 \cge \theta$ holds.
  The equation $(\theta_1 \circ \theta_2) \quot \theta = (\theta_1 \quot \theta) \circ (\theta_2 \quot \theta)$
  easily follows from the definition of the quotient of congruences.
\end{proof}

\begin{theorem}\label{thm:prod-cong}
  Let $A = (A, d)$ be a complete metric algebra
  and $\theta_1$, $\theta_2$ be congruences on $A$.
  The canonical homomorphism $f \colon A \to A \quot \theta_1 \times A \quot \theta_2$
  is isomorphic if the following conditions hold:
  \begin{enumerate}
    \item $\theta_1 \cmeet \theta_2 = d$.
    \item $\theta_1 \cjoin \theta_2 = 0$.
    \item $\theta_1$ and $\theta_2$ are permutable.
  \end{enumerate}
\end{theorem}
\begin{proof}
  For $a, b \in A$, we have $d(f(a), f(b)) = \max(\theta_1(a, b), \theta_2(a, b)) = d(a, b)$, so $f$ is isometric.
  We show that $f$ is surjective. Suppose $a_1, a_2 \in A$.
  Since $\theta_1 \circ \theta_2 = \theta_1 \cjoin \theta_2 = 0$,
  there exists a sequence $(c^n)_{n=1}^{\infty}$ in $A$ such that
  $\theta_1(a_1, c^n) + \theta_2(c^n, a_2) \le 2^{-n}$ for each $n \in \omega$.
  Since $\theta_i(c^n, c^m) \le \theta_i(c^n, a_i) + \theta_i(a_i, c^m) \le 2^{-n+1}$
  for $i = 1, 2$ and $n \lt m$, we also have $d(c^n, c^m) \le 2^{-n+1}$ by
  $\theta_1 \cmeet \theta_2 = d$. Therefore $(c^n)_{n=1}^{\infty}$ is a Cauchy sequence
  in $(A, d)$ and has a convergent point $c = \lim_{n \to \infty} c^n$.
  Since $\theta_i(a_i, c) \le \theta_i(a_i, c^n) + \theta_i(c^n, c) \le 2^{-n} + d(c^n, c) \to 0$
  (as $n \to \infty$),
  we conclude $\theta_1(a_1, c) = \theta_2(a_2, c) = 0$, that is, $f(c) = ([a_1], [a_2])$.
\end{proof}

By inductively applying Theorem~\ref{thm:prod-cong},
we acquire a slightly generalized version of the theorem for the arbitrary finite cases.

\begin{corollary}
  Let $A = (A, d)$ be a complete metric algebra
  and $(\theta_i)_{i=1}^{n}$ be a family of congruences on $A$.
  The canonical homomorphism $f \colon A \to \prod_{i=1}^n A \quot \theta_i$
  is isomorphic if the following conditions hold:
  \begin{enumerate}
    \item $\bigcurlywedge_{i=1}^n \theta_i = d$.
    \item $(\bigcurlywedge_{i=1}^{k-1} \theta_i) \cjoin \theta_k = 0$ for each $k = 2, \ldots, n$.
    \item $\theta_i$ and $\theta_j$ are permutable for each $i \neq j$.
  \end{enumerate}
\end{corollary}
\begin{proof}
  The proof is by induction on $n$.
  The case $n = 1$ is obviously valid.
  Let us suppose that the proposition holds for $n$; then we prove it for $n+1$.
  Assume $(\theta_i)_{i=1}^{n+1}$ is a family of congruences satisfying the conditions.
  Let $\rho_1 = \bigcurlywedge_{i=1}^n \theta_i$ and $\rho_2 = \theta_{n+1}$.
  Since $\rho_1 \cmeet \rho_2 = d$, $\rho_1 \cjoin \rho_2 = 0$
  and $\rho_1$ and $\rho_2$ are permutable, then
  $A$ is canonically isomorphic to $A \quot \rho_1 \times A \quot \rho_2$.
  The family $(\theta_i \quot \rho_1)_{i=1}^{n}$ of congruences on $A \quot \rho_1$
  satisfies the assumption of the proposition, therefore
  $A \quot \rho_1$ is isomorphic to $\prod_{i=1}^{n} A \quot \theta_i$ by the induction hypothesis.
\end{proof}

\begin{remark}
  The completeness of $A$ is essential.
  Let $\Sigma = \emptyset$.
  Consider $X = [0, 1]^2 \setminus \{(0, 0)\}$ with the supremum metric
  and congruences $\theta_i((x_1, x_2),(y_1, y_2)) = \abs{x_i - y_i}$ for $i = 1, 2$.
  These congruences satisfy $\theta_1 \circ \theta_2 = \theta_2 \circ \theta_1 = 0$;
  let $x = (x_1, x_2)$ and $y = (y_1, y_2)$ in $X$.
  For $\e \ge 0$, take $z = (x_1 + \e, y_2 + \e) \in X$ and then we have
  $\theta_1 \circ \theta_2 (x, y) \le \theta_1(x, z) + \theta_2(z, y) = 2 \e$.
  Letting $\e \to 0$, we get $\theta_1 \circ \theta_2 (x, y) = 0$.
  Similarly we have $\theta_2 \circ \theta_1 = 0$.
  However $X \quot \theta_i \isom [0, 1]$ and $X$ is not isometric
  to $[0, 1]^2$.
\end{remark}

\section{Syntax and Logic}
\label{ch:logical-metric}

So far we have explained the model theoretic aspect of metric algebras.
In this section, we give the syntax to describe properties of metric algebras,
and prove some basic theorems such as a weak form of the compactness theorem.

\subsection{Syntax for Metric Algebra}
\label{sec:metric-equation}

We use indexed equations $s =_\e t$ for atomic formulas in the theory of metric algebras,
differently from usual equations $s = t$ in the classical case.

\begin{definition}[\cite{Weaver1995, Mardare2016}]
  Let $X$ be a variable set.
  \begin{itemize}
    \item
    A \emph{metric equation} (also called an \emph{atomic inequality} \cite{Weaver1995})
    \emph{over $X$}
    is a formula of the form $s =_{\e} t$
    where $s, t \in \Term_{\Sigma} X$ and $\e \ge 0$.

    A \emph{metric implication over $X$} is a formula of
    the form $\bigwedge_{i=1}^n s_i =_{\e_i} t_i \to s =_{\e} t$
    where $s_i =_{\e_i} t_i$ and $s =_{\e} t$ are metric equations over $X$.
    We will identify a metric equation with a metric implication where $n = 0$.

    A \emph{basic quantitative inference over $X$} is a metric implication
    where $s_i$ and $t_i$ are restricted to variables.
    A \emph{$\kappa$-basic quantitative inference} is
    its generalization that allows infinitely many assumptions smaller than $\kappa$.

    \item   Given a metric algebra $A$, a metric equation $s =_{\e} t$ over $X$,
      and a map $v \colon X \to A$,
      we say \emph{$A$ satisfies $s =_{\e} t$ under $v$},
      denoted by $A, v \models s =_{\e} t$, if we have $d(\sem{s}_v, \sem{t}_v) \le \e$.
      We simply say \emph{$A$ satisfies $s =_{\e} t$},
      denoted by $A \models s =_{\e} t$,
      when $A, v \models s =_{\e} t$ holds for any $v \colon X \to A$.
      These notions are similarly defined for metric implications.

    \item
      Let $\cat{K}$ be a class of metric algebras and $\Phi \cup \{ \phi \}$
      be a set of metric implications.
      We write $\cat{K} \models \phi$ if $B \models \phi$ holds for any $B \in \cat{K}$.
      We also define $A \models \Phi$ and $\cat{K} \models \Phi$ similarly.

    \item Let $\Delta \cup \{s =_\e t\}$ be a set of metric equations over $X$.
      We write $\Delta \models_{\cat{K}} s =_\e t$ if, for $A \in \cat{K}$ and a map $v \colon X \to A$
      with $A, v \models \Delta$, we have $A, v \models s =_\e t$.

    \item  Given a class $\cat{V}$ of metric algebras and a set $\Phi$ of metric implications,
      we define the class $\cat{V}(\Phi)$ by $\cat{V}(\Phi) = \set{A \in \cat{V}}{A \models \Phi}$,
      called \emph{the class defined in $\cat{V}$ by $\Phi$}.
      When $\cat{V} = \cat{M}$, we simply call it \emph{the class defined by $\Phi$}
  \end{itemize}
\end{definition}

Given a pseudometric $\theta$ on a set $X$, we identify $\theta$ with a set $E_\theta$ of
metric equations over $X$ defined by $E_\theta = \set{x =_\e y}{x, y \in X,\, \theta(x, y) \le \e}$.
This view is consistent with the reversed pointwise order on $\Con A$:
we have $\theta_1 \cle \theta_2$ if and only if $\theta_1 \subset \theta_2$ holds.

\subsection{Presentation and Free Algebra}

As in classical universal algebra, a metric algebra can be presented by generators and relations
in a given class $\cat{K}$. As the special case, we give the construction of
$\cat{K}$-free algebras.

\begin{definition}
  \label{presentation-malg}
  A \emph{presentation of a metric algebra} is a pair $(X, \Delta)$
  where $X$ is a set and $\Delta$ is a set of metric equations over $X$.

  Let $\cat{K}$ be a class of metric algebras. Given a presentation $(X, \Delta)$,
  the \emph{metric algebra defined by $(X, \Delta)$ in $\cat{K}$}
  is a metric algebra $\Free_{\cat{K}} (X, \Delta) = \Term_{\Sigma} X \quot \theta_\Delta$
  where $\theta_\Delta$ is the smallest $\bbS(\cat{K})$-congruential pseudometric that contains $\Delta$.
  It is equipped with a map $\eta \colon X \to \Free_{\cat{K}} (X, \Delta)$
  defined by $\eta(x) = [x]$, which is called its \emph{unit}.

  We write $\Free_{\cat{K}} X$ when $\Delta = \emptyset$,
  which is called \emph{the $\cat{K}$-free algebra over $X$},
  and write $\Free_{\cat{K}} (X, d)$ for a metric space $(X, d)$
  when $\Delta = \set{x =_\e y}{d(x, y) \le \e}$,
  which is called \emph{the $\cat{K}$-free algebra over $(X, d)$}.
\end{definition}

\begin{lemma}
  \label{lem:presentation}
  Let $s =_\e t$ be a metric equation. In Definition~\ref{presentation-malg}:
  \begin{enumerate}
    \item $\Free_{\cat{K}} (X, \Delta), \eta \models \Delta$ holds.
    \item For any $A \in \cat{K}$ and a map $f \colon X \to A$
      where $A, f \models \Delta$ holds, there exists
      a unique homomorphism $h \colon \Free_{\cat{K}} (X, \Delta) \to A$ such that
      $h \circ \eta = f$.
    \item $\Free_{\cat{K}} (X, \Delta), \eta \models s =_\e t$
      if and only if $\Delta \models_{\cat{K}} s =_{\e} t$.
    \item If $\cat{K}$ is a prevariety, then
      $\Free_{\cat{K}} (X, \Delta)$ belongs to $\cat{K}$.
  \end{enumerate}
\end{lemma}
\begin{proof}
  (1) It is obvious from
  $\Free_{\cat{K}} (X, \Delta) = \Term_{\Sigma} X \quot \theta_\Delta$ and $\Delta \subset \theta_\Delta$.

  (2) Let us consider $f^{\sharp} \colon \Term_{\Sigma} X \to A$.
  Since $\Delta \subset \ker(f^{\sharp})$ holds by assumption and
  $\ker(f^{\sharp})$ is $\bbS(\cat{K})$-congruential,
  we have $\theta_\Delta \cle \ker(f^{\sharp})$.
  By Proposition~\ref{first-isom-thm}, there exists a unique
  homomorphism $h \colon \Free_{\cat{K}} (X, \Delta) \to A$
  such that $h \circ \eta = f$.

  (3) (if) Assume $\Delta \models_{\cat{K}} s =_{\e} t$.
  Since $\Free_{\cat{K}} (X, \Delta) \models_{\eta} \Delta$ holds by (1),
  then we conclude $\Free_{\cat{K}} (X, \Delta) \models_{\eta} s =_{\e} t$.
  (only if)
  Let $A \in \cat{K}$ and $v \colon X \to A$ be a map with $A, v \models \Delta$.
  By (2), we have a homomorphism $h \colon \Free_{\cat{K}} (X, \Delta) \to A$
  such that $h \circ \eta = v$.
  Therefore
  $d(\sem{s}_v, \sem{t}_v) \le d(h(\sem{s}_\eta), h(\sem{t}_\eta))
  \le d(\sem{s}_\eta, \sem{t}_\eta) \le \e$.

  (4) Directly follows from Corollary~\ref{cor:cong-complat}.
\end{proof}

\subsection{Weak Compactness Theorem}
\label{sec:weak-compactness}

We do not have the full version of the compactness theorem.
There are two restrictions:
we restrict ourselves to metric equations,
and a finite subset of the assumptions is chosen only for
each perturbation of the conclusion by $\e \gt 0$.

\begin{theorem}[Weak compactness]
  \label{thm:weak-compactness}
  Let $\cat{K}$ be a continuous class of metric algebras and
  $\Delta \cup \{ s =_\e t \}$ be a set of metric equations over $X$.
  If $\Delta \models_{\cat{K}} s =_\e t$,
  then for any $\e' \gt \e$ there exists
  a finite subset $\Delta_0 \subset \Delta$ such that
  $\Delta_0 \models_{\cat{K}} s =_{\e'} t$.
\end{theorem}
\begin{proof}
  We prove the theorem by contradiction.
  Suppose that there exists $\delta \gt 0$ such that, for any finite subset $\Gamma \subset \Delta$,
  we have $A_{\Gamma} \in \cat{K}$ and a map $v_{\Gamma} \colon X \to A_{\Gamma}$
  where $A_{\Gamma}, v_{\Gamma} \models \Gamma$ and
  $A_{\Gamma}, v_{\Gamma} \not\models s =_{\e'} t$ hold.

  Let $I$ be the set of finite subsets of $\Delta$.
  We define $J_{\Gamma} = \set{\Gamma' \in I}{\Gamma \subset \Gamma'}$
  for each $\Gamma \in I$
  and $\calB = \set{J_{\Gamma}}{\Gamma \in I}$.
  Since $\calB$ satisfies the finite intersection property,
  there exists an ultrafilter $\calU$ containing $\calB$.
  Let $A = \prod_{\Gamma}^{\calU} A_{\Gamma}$ be the ultraproduct of metric algebras,
  and $v \colon X \to A$ be a map defined by $v(x) = [v_i(x)]_i$.
  Then $A, v \models \Delta$ and
  $d(\sem{s}_v, \sem{t}_v)  \ge \e' \gt \e$, hence $A, v \not\models s =_{\e} t$,
  which contradicts $A \in \cat{K}$ and the assumption.
\end{proof}

\begin{corollary}
  Let $\cat{K}$ be a continuous class of metric algebras,
  $(X, \Delta)$ be a presentation and $\eta \colon X \to \Free_{\cat{K}} (X, \Delta)$
  be the unit.
  Then $d(\sem{s}_\eta, \sem{t}_\eta) \le \e$ holds if and only if
  for any $\e' \gt \e$
  there is a finite subset $\Delta_0 \subset \Delta$ such that
  $\Delta_0 \models_{\cat{K}} s =_{\e'} t$.
\end{corollary}
\begin{proof}
  By Lemma~\ref{lem:presentation} (3) and Theorem~\ref{thm:weak-compactness}.
\end{proof}

Using this weak version of the compact theorem,
we can show that \emph{continuous quasivarieties} in \cite{Khudyakov2003}
(simply called \emph{quasivarieties} in \cite{Weaver1995}) are
expectedly quasivarieties that are continuous in our terminology.

\begin{definition}[\cite{Weaver1995}]
  Given a class $\cat{K}$ of metric algebras and
  a metric implication $\phi \equiv \bigland_{i=1}^{n} s_i =_{\e_i} t_i \to s =_{\e} t$,
  we say \emph{$\cat{K}$ satisfies $\phi$ equicontinuously}
  if, for any $\e' \gt \e$, there exists $\delta \gt 0$ such that
  $\cat{K} \models \bigland_{i=1}^{n} s_i =_{\e_i + \delta} t_i \to s =_{\e'} t$.
\end{definition}

\begin{proposition}
  \label{cor:cont-class-equicont}
  Let $\cat{K}$ be a continuous class
  and $\phi \equiv \bigland_{i=1}^{n} s_i =_{\e_i} t_i \to s =_{\e} t$
  be a metric implication.
  If $\cat{K}$ satisfies $\phi$,
  then $\cat{K}$ satisfies $\phi$ equicontinuously.
\end{proposition}
\begin{proof}
  Let $\Delta = \set{s_i =_{\e_i + \delta} t_i}{\delta \gt 0, \, i = 1, \ldots, n}$.
  We have $\Delta \models_{\cat{K}} s =_{\e} t$ by assumption.
  Given $\e' \gt 0$, by Theorem~\ref{thm:weak-compactness},
  there exists a finite subset $\Delta_0 \subset \Delta$ such that
  $\Delta_0 \models_{\cat{K}} s_0 =_{\e'} t_0$.
  Since $\Delta_0$ is finite, we can take the minimum $\delta \gt 0$ that arises in $\Delta_0$.
  And then we have $\cat{K} \models \bigland_{i} s_i =_{\e_i + \delta} t_i \to s =_{\e'} t$.
\end{proof}

\begin{definition}[\cite{Khudyakov2003}]
  A \emph{continuous family of metric implications} is
  a set $\Phi$ of metric implications that satisfies the following conditions:
  \begin{itemize}
    \item For each $\sigma \in \Sigma$, the formula
    $\vec{x} =_0 \vec{y} \to \sigma(\vec{x}) =_0 \sigma(\vec{y})$
    belongs to $\Phi$.
    \item If $\bigwedge_{i=1}^n s_i =_{\e_i} t_i \to s =_{\e} t$ belongs to $\Phi$
      and $\e' \gt \e$, then there exists $\delta \ge 0$ such that
      $\bigwedge_{i=1}^n s_i =_{\e_i + \delta} t_i \to s =_{\e'} t$ also belongs to $\Phi$.
  \end{itemize}
\end{definition}

\begin{lemma}
  \label{lem:cont-cont}
  Let $\Phi$ be a continuous family of metric implications.
  Then $\cat{V}(\Phi)$ is closed under reduced products.
\end{lemma}
\begin{proof}
  Let $\calF$ be a filter on $I$
  and $(A_i, d_i)_{i \in I}$ be a family of metric algebras with $A_i \models \Phi$.
  Let $\phi \equiv \bigland_{k=1}^{n} s_k =_{\e_k} t_k \to s =_{\e} t$ be
  a metric implication over $X$ that belongs to $\Phi$, and we show $A \models \phi$
  where $A = \prod_{i}^{\calF} A_i$.

  Let $(v_i \colon X \to A_i)_i$ be a family of maps, and $v \colon X \to A$
  be a map defined by $v(x) = [v_i(x)]_i$.
  Assume $A, v \models s_k =_{\e_k} t_k$ for each $k = 1, \ldots, n$,
  and let us fix $\e' \gt \e$. Since $\Phi$ is a continuous family of metric implications,
  there exists $\delta \gt 0$ such that
  $\phi' \equiv \bigland_{k=1}^{n} s_k =_{\e_k + \delta} t_k \to s =_{\e'} t$ belongs to $\Phi$.
  By the definition of reduced product,
  there exists $J \in \calF$ such that,
  for each $k = 1, \ldots, n$, we have $\sup_{i \in J} d_i(\sem{s_k}_{v_i}, \sem{t_k}_{v_i}) \le \e_k + \delta$.
  Since $A_i, v_i \models \phi'$, we also have
  $\limsup_{i \to \calF} d_i(\sem{s}_{v_i}, \sem{t}_{v_i}) \le \sup_{i \in J} d_i(\sem{s}_{v_i}, \sem{t}_{v_i}) \le \e'$.
  Letting $\e' \to \e$, we conclude $A, v \models s =_\e t$.
\end{proof}

\begin{proposition}
  \label{equiv-conti-quasi}
  Let $\cat{K}$ be a quasivariety.
  The followings are equivalent:
  \begin{enumerate}
    \item $\cat{K}$ is a continuous quasivariety.
    \item The set $\Phi_{\cat{K}}$ of metric implications that holds in $\cat{K}$ is continuous.
    \item $\cat{K}$ is defined by a continuous family of metric implications.
  \end{enumerate}
\end{proposition}
\begin{proof}
$(1 \Rightarrow 2)$ It follows from Proposition~\ref{cor:cont-class-equicont}.

$(2 \Rightarrow 3)$ $\cat{K}$ is defined by $\Phi_{\cat{K}}$.

$(3 \Rightarrow 1)$ Directly follows from Lemma~\ref{lem:cont-cont}.
\end{proof}

%

\subsection{Generalized Metric Inequality}
\label{sec:met-ineq}

In Subsection~\ref{sec:weak-compactness}, we saw the ultraproduct construction
preserves properties described by a continuous family of metric implications.
We will see that some richer properties are preserved by ultraproducts.

\begin{definition}
  A \emph{(generalized) metric inequality over a set $X$} is a tuple $(f; \vec{s}, \vec{t})$ of
  a continuous function $f \colon (\bbRp)^n \to \bbRp$ and
  terms $\vec{s} = (s_1, \ldots, s_n)$, $\vec{t} = (t_1, \ldots, t_n)$ over $X$,
  denoted by $f(d(s_1, t_1), \ldots, d(s_n, t_n)) \ge 0$.

  Given a metric algebra $A$ and a map $v \colon X \to A$,
  \emph{the metric inequality $f(d(s_1, t_1), \ldots, d(s_n, t_n)) \ge 0$ holds under $v$},
  if the following condition holds.
  \[
    f(d(\sem{s_1}_v, \sem{t_1}_v), \ldots, d(\sem{s_n}_v, \sem{t_n}_v)) \ge 0 \quad.
  \]

  Other expressions such as
  $f(d(s_1, t_1), d(s_2, t_2)) \le g(d(s'_1, t'_1), d(s'_2, t'_2))$
  and $f(d(s_1, t_1), d(s_2, t_2)) = 0$ are defined and interpreted naturally.
\end{definition}

\begin{theorem}
  \label{th:ultraprod-met-ineq}
  Let $\phi \equiv f(d(s_1, t_1), \ldots, d(s_n, t_n)) \ge 0$ be a metric inequality,
  $\calU$ be an ultrafilter on $I$, and $(A_i)$ be a family of metric algebras.
  If $A_i$ satisfies $\phi$ for any $i \in I$,
  the ultraproduct $\prod_{i}^{\calU} A_i$ also satisfies $\phi$.
\end{theorem}
\begin{proof}
  Let $(v_i \colon X \to A_i)_{i \in I}$ be a family of maps,
  and $v \colon X \to \prod_{i}^{\calU} A_i$ be a map defined by $v(x) = [v(x)]_i$.
  Let us define $x_k^i = d(\sem{s_k}_{v_i}, \sem{t_k}_{v_i})$
  and $\gamma_k = d(\sem{s_k}_v, \sem{t_k}_v)$,
  and assume $A_i, v \models \phi$ for each $i \in I$.
  By Proposition~\ref{cont-limit} and $\gamma_k = \lim_{i \to \calU} x_k^i$,
  we have
  $f(\gamma_1, \ldots, \gamma_n) = \lim_{i \to \calU} f(x_1^i, \ldots, x_n^i) \ge 0$.
\end{proof}

An immediate application is on the class of inner product spaces.
An inner product space is equipped with the norm determined by its inner product.
A classical result of functional analysis states that a norm that satisfies a certain equation
comes from an inner product. Then we can apply Theorem~\ref{th:ultraprod-met-ineq}
and prove that the class of inner product spaces is closed under ultraproducts.
See \cite{Heinrich1980, Li2005} for more examples from functional analysis and operator algebra.

\begin{example}
  For the signature of normed vector space,
  \[
    \norm{x + y}^2 + \norm{x - y}^2 = 2 (\norm{x}^2 + \norm{y}^2)
  \]
  is a metric inequality, where $\norm{z}$ is a shorthand for $d(z, 0)$.
  This metric inequality characterizes the class of inner product spaces \cite{JordanN1935}.
\end{example}

\begin{corollary}[\cite{Li2005}]
  Ultraproducts of inner product spaces are inner product spaces.
  Moreover ultraproducts of Hilbert spaces are Hilbert spaces.
\end{corollary}
\begin{proof}
  By Theorem~\ref{th:ultraprod-met-ineq} and Proposition~\ref{thm:uprod-comp-comp}.
\end{proof}

\section{Variety Theorem}
\label{chp:variety}

Now we prove the variety theorems of metric algebras.

\subsection{Basic Closure Properties}

We know the following closure properties of
classes defined by a certain formula.

\begin{proposition}[\cite{Weaver1995, Mardare2017}]
  \label{basic-closure}
  Let $\phi$ be a metric implication.
  \begin{enumerate}
    \item The class $\cat{M}(\phi)$ is closed under subalgebras, products.
    \item If $\phi$ is a quantitative basic inference, then $\cat{M}(\phi)$ is
      closed under $\omega$-reflexive quotients, and then closed under reflexive quotients.
    \item If $\phi$ is a metric equation, then $\cat{M}(\phi)$ is
      closed under quotients.
  \end{enumerate}
\end{proposition}

Our goal is to prove the converse of this result: if a class of metric algebras
is closed under some constructions, it is defined by
a certain class of formulas.

\subsection{Strict Variety Theorem}

First we give a very simple version of metric variety theorems.
As we will see, this formulation is very naive, which excludes
normed vector spaces as example.

\begin{definition}
  \label{def:strict-variety}
  A class of metric algebras is called a \emph{strict variety}
  (also called a \emph{1-variety} in \cite{Mardare2017})
  if it is defined by a set of metric equations.
\end{definition}

The proof of the strict variety theorem is
almost analogous to the classical case in \cite{Burris1981};
we use congruential pseudometrics instead of congruences.

\begin{theorem}[\cite{Hino2016}]
  \label{hsp-theorem}
  A class $\cat{K}$ of metric algebras is a strict variety
  if and only if $\cat{K}$ is closed under products, subalgebras and quotients.
\end{theorem}
\begin{proof}
  (only if) Directly follows from Proposition~\ref{basic-closure}.

  (if)
  Let $E$ be the set of metric equations that hold in $\cat{K}$.
  Since $\cat{K} \subset \cat{M}(E)$ is trivial, we only have to
  show $\cat{M}(E) \subset \cat{K}$.
  Let $A$ be a metric algebra that satisfies $E$,
  and $X$ be its underlying set.
  Let $f \colon \Term_{\Sigma} X \to A$ be the homomorphic extension of the identity map
  and $\pi \colon \Term_{\Sigma} X \to \Free_{\cat{K}} X$ be the canonical projection.
  By Lemma~\ref{lem:presentation} we have $\Free_{\cat{K}} X \in \cat{K}$.
  Since $f$ and $\pi$ are surjective, by Proposition~\ref{first-isom-thm}, it suffices to show that
  $d(f(s), f(t)) \le d(\pi(s), \pi(t))$ for any $s, t \in \Term_\Sigma X$.

  Assume $d(\pi(s), \pi(t)) \le \e$, that is, $\Free_{\cat{K}} X, \pi \models s =_\e t$.
  By Lemma~\ref{lem:presentation} (3), we have $\cat{K} \models s =_\e t$.
  Since $A$ satisfies all metric equations that hold in $\cat{K}$,
  we have $A \models s =_\e t$ and especially $d(f(s), f(t)) \le \e$,
  which concludes the proof.
\end{proof}

Applying Theorem~\ref{hsp-theorem},
we can prove that the class of normed vector spaces is not a variety of metric algebras
for the signature of vector space.

\begin{example}
  \label{eg:nvs-not-0variety}
  For the signature $\Sigma = \{{+}, 0, (\lambda {\cdot}) \}_{\lambda \in \bbR}$,
  the class $\cat{N}$ of normed vector spaces is a \emph{quasivariety}
  of metric algebras \cite{Weaver1995, Khudyakov2003},
  but it is not a strict variety. Indeed consider $\bbR \in \cat{N}$ and
  let $R'$ be a metric algebra that has the same algebraic structure as $\bbR$
  but whose metric is defined by $d(x, y) = \abs{\tanh(y) - \tanh(x)}$.
  The identity map $f \colon \bbR \to R'$ is a quotient
  while $R' \not\in \cat{N}$. Therefore $\cat{N}$ is not closed under quotient,
  hence not a strict variety.
\end{example}

The class of normed vector space is a prototypical example of classes of metric algebras,
but Example~\ref{eg:nvs-not-0variety} showed that it cannot be expressed by metric equations.
To deal with such classes, we need to use more expressive formulas.

\newpara

We can extend the strict variety theorem to the quantitative case.
For a generality, we give the notion of \emph{variety relative to $\cat{V}$}
(see \cite{Gorbunov1998} for the classical case)
and deal with the quantitative case as its particular case.

\begin{definition}
  \label{def:strict-variety-rel}
  A class of metric algebras is a \emph{strict variety relative to $\cat{V}$}
  if it is defined in $\cat{V}$ by a set of metric equations.
\end{definition}

\begin{theorem}
  \label{hsp-theorem-rel}
  Let $\cat{V}$ be a prevariety.
  A class $\cat{K} \subset \cat{V}$ of metric algebras is a strict variety
  relative to $\cat{V}$
  if and only if $\cat{K}$ is closed under products, subalgebras and $\cat{V}$-quotients.
\end{theorem}
\begin{proof}
  By Proposition~\ref{classop-comp}, $\bbH(\cat{K})$ is closed under quotients, subalgebras and products.
  Then by Theorem~\ref{hsp-theorem},
  there exists a set $E$ of metric equations
  such that $\cat{M}(E) = \bbH(\cat{K})$.
  Thus by assumption $\cat{M}(E) \cap \cat{V} = \bbH(\cat{K}) \cap \cat{V} = \cat{K}$,
  which concludes the proof.
\end{proof}

\begin{corollary}[\cite{Hino2016}]
  \label{strict-qvar}
  A class $\cat{K}$ of quantitative algebras is a strict variety
  relative to $\cat{Q}$
  if and only if $\cat{K}$ is closed under products, subalgebras and $\cat{Q}$-quotients.
\end{corollary}

\subsection{Continuous Variety Theorem}
\label{sec:cont-variety}

We give a characterization of classes defined by basic quantitative inferences.

Mardare et al.\ give a solution for this characterization problem in \cite{Mardare2017}.

\begin{theorem}[\cite{Mardare2017}]
  For a cardinal $\kappa \le \aleph_1$,
  a class of metric algebras is defined by a set of $\kappa$-basic quantitative inferences
  if and only if it is closed under subalgebras, products and $\kappa$-reflexive quotients.
\end{theorem}

Differently from their result, our goal is to prove the continuous version.
In this case, the size condition is included in the continuity assumption.

\begin{theorem}[Continuous variety theorem]
  A class $\cat{K}$ of metric algebras is defined by a continuous family of basic quantitative inferences
  if and only if it is closed under subalgebras, products, reflexive quotients and ultraproducts.
\end{theorem}
\begin{proof}
  (only if) Directly follows from Proposition~\ref{basic-closure} and Lemma~\ref{lem:cont-cont}.

  (if)
  Let $E$ be the set of basic quantitative inferences that hold in $\cat{K}$.
  We show $\cat{M}(E) \subset \cat{K}$.
  Let $A$ be a metric algebra that satisfies $E$,
  and $(A, d)$ be its underlying metric space.
  Let $f \colon \Term_{\Sigma} A \to A$ be the homomorphic extension of the identity map
  and $\pi \colon \Term_{\Sigma} A \to \Free_{\cat{K}} (A, d)$ be the canonical projection.
  It suffices to show that (1) $d(f(s), f(t)) \le d(\pi(s), \pi(t))$ for any $s, t \in \Term_\Sigma A$,
  and (2) $d(f(a), f(b)) = d(\pi(a), \pi(b))$ for any $a, b \in A \subset \Term_\Sigma A$.

  (1) Assume $d(\pi(s), \pi(t)) \le \e$, that is, $\Free_{\cat{K}} (A, d), \pi \models s =_\e t$.
  By Lemma~\ref{lem:presentation} (3), we have $\Delta \models_\cat{K} s =_\e t$,
  where $\Delta = \set{a =_{\delta} b}{a, b \in A,\, d(a, b) \le \delta}$.
  Given $\e' \gt \e$, by Theorem~\ref{thm:weak-compactness} there exists a finite subset $\Delta_0 \subset \Delta$
  such that $\Delta_0 \models_\cat{K} s =_{\e'} t$.
  Since $A$ satisfies all quantitative basic inferences that holds in $\cat{K}$,
  we have $A \models \bigwedge \Delta_0 \to s =_{\e'} t$.
  Since $A, f \models \Delta_0$ by the definition of $\Delta_0 \subset \Delta$,
  we have $A, f \models s =_{\e'} t$, that is, $d(f(s), f(t)) \le \e'$.
  Letting $\e' \to \e$, we conclude $d(f(s), f(t)) \le \e$.

  (2) We only have to show that $d(f(a), f(b)) \ge d(\pi(a), \pi(b))$. Let us assume $d(f(a), f(b)) \le \e$.
  Since $f$ is an identity on $A$, we have $d(a, b) \le \e$. This means $a =_{\e} b \in \Delta$
  and then $d(\pi(a), \pi(b)) \le \e$.

  Therefore the class $\cat{K}$ is defined by $E$, hence a quasivariety.
  By Proposition~\ref{equiv-conti-quasi}, the family $E$ is moreover continuous,
  which concludes the proof.
\end{proof}

\begin{corollary}
  A class $\cat{K}$ of quantitative algebras is defined by a continuous family of
  basic quantitative inferences in $\cat{Q}$
  if and only if $\cat{K}$ is closed under products, subalgebras, ultraproducts and
  reflexive $\cat{Q}$-quotients.
\end{corollary}
\begin{proof}
  Exactly the same as Corollary~\ref{strict-qvar}.
\end{proof}

\section{Conclusions and Future Work}
\label{ch:conclusions}

We developed a general theory of metric and quantitative algebra
from the viewpoint of universal algebra.
We investigated the lattices of congruential pseudometrics
on a metric algebra, and proved the metric variants of the variety theorem
by using their structure.

Our work is different from \cite{Mardare2017} because
we aim at continuous classes of metric and quantitative algebras,
following the work by Weaver \cite{Weaver1995} and Khudyakov \cite{Khudyakov2003}.
This design choice seems to be natural since the continuity of classes of metric algebras
can be understood as a sort of closedness in the topological sense, hence a sort of robustness.
Moreover our result is mainly on general metric algebras rather than quantitative algebras,
which enables our theory to include examples from functional analysis and operator algebra.

We did not pursue the connection to the category theoretic treatments
of universal algebra: Lawvere theory, monad and orthogonality.

The theory of quantitative algebra can be viewed as a special case of enriched Lawvere theory.
More specifically, it is the discrete Lawvere theory \cite{HylandP2006} enriched by the category of metric spaces.
Here the adjective \emph{discrete} means that we only consider operations
whose arities are natural numbers, while in enriched Lawvere theory
an operation whose arity is a finite metric space is allowed.
It would be possible to give a syntax and prove the variety theorem for that situation.

The use of monads and Eilenberg-Moore categories is another way
to deal with equational theories in category theory.
Mardare et al.\ showed that a class of quantitative algebras defined by basic quantitative inferences
induce a monad on the category of metric spaces.
The next problem is whether the class of quantitative algebras is monadic.

It would also be interesting to check whether our work is an instance of
the categorical variety theorem formulated by Ad\'{a}mek et al.\ in \cite{JoyOfCats}.
Our theory seems to implicitly use the orthogonal factorization system on the category of metric algebras
that consists of embeddings and quotients.
But there is another factorization system: closed embeddings and dense maps.
The natural question is what kind of variety theorems is acquired
if we use this factorization system instead of embeddings and quotients.

The metric structures on free algebras are also yet to be investigated.
For example, we could investigate whether the free algebra on a metric space is complete, or compact
for a given axiom of metric algebras.




\bibliographystyle{plain}
\bibliography{ref}

\end{document}